\documentclass[11pt]{article}
\usepackage[T1]{fontenc}
\usepackage{amsmath,amsthm,amsfonts,tikz,bm,tipa}

\usetikzlibrary{arrows.meta}
\usepackage{amssymb}
\usepackage[hidelinks]{hyperref}

\newtheorem{theorem}{Theorem}[section]
\newtheorem{lemma}[theorem]{Lemma}
\newtheorem{prop}[theorem]{Proposition}
\newtheorem{corollary}[theorem]{Corollary}

\DeclareMathAlphabet{\mathbfit}{OML}{cmm}{b}{it}
\DeclareMathAlphabet{\mathbfcal}{OMS}{cmsy}{b}{n}

\usepackage[margin=1in]{geometry}

\DeclareSymbolFont{kplargesymbols}{OMX}{jkp}{m}{n}
\DeclareMathAccent{\widearc}{\mathalpha}{kplargesymbols}{"86}

\DeclareMathOperator{\Iso}{\mathrm{Isom}}

\DeclareMathOperator{\trr}{\mathrm{tr}}
\DeclareMathOperator{\arcosh}{\mathrm{arcosh}}

\DeclareMathOperator{\sig}{\mathrm{sig}}

\begin{document}

\newcommand{\R}{\mathbb{R}}
\newcommand{\C}{\mathbb{C}}
\newcommand{\hh}{\mathcal{H}}
\newcommand{\PGL}{\mathrm{PGL}}
\newcommand{\PSL}{\mathrm{PSL}}
\newcommand{\Ann}{\mathrm{Ann}}
\newcommand{\Gr}{\mathrm{Gr}}
\newcommand{\Herm}{\mathrm{Herm}}

\newcommand{\la}{\langle}
\newcommand{\ra}{\rangle}
\newcommand{\pr}{\prime}
\newcommand{\oline}{\overline}
\newcommand{\triag}{\triangle}

\newcommand{\cl}{\bm{\ell}}
\newcommand{\ca}{\mathbfit{a}}
\newcommand{\cb}{\mathbfit{b}}
\newcommand{\cc}{\mathbfit{c}}
\newcommand{\cd}{\mathbfit{d}}
\newcommand{\cn}{\mathbfit{n}}
\newcommand{\cm}{\mathbfit{m}}
\newcommand{\ch}{\mathbfit{h}}
\newcommand{\ci}{\mathbfit{i}}
\newcommand{\ce}{\mathbfit{e}}
\newcommand{\cp}{\mathbfit{p}}
\newcommand{\cs}{\mathbfit{s}}

\newcommand{\M}{\mathbfit{M}}
\newcommand{\CC}{\mathbfit{C}}
\newcommand{\GG}{\mathbfit{G}}
\newcommand{\PP}{\mathbfit{P}}
\newcommand{\LL}{\mathbfit{L}}

\newcommand{\gA}{\mathbfcal{A}}
\newcommand{\gB}{\mathbfcal{B}}
\newcommand{\gC}{\mathbfcal{C}}
\newcommand{\gP}{\mathbfcal{P}}
\newcommand{\gQ}{\mathbfcal{Q}}
\newcommand{\gL}{\mathbfcal{L}}
\newcommand{\gI}{\mathbfcal{I}}
\newcommand{\gH}{\mathbfcal{H}}
\newcommand{\gO}{\mathbfcal{O}}

\title{On Ceva's and Menelaus's Theorems for a M\"obius triangle}
\author{Ivan Livinsky}
\date{}
\maketitle

\begin{abstract}
We generalize the classical Ceva's and Menelaus's theorems to curvilinear
triangles bounded by circular arcs (Figure~\ref{fig:ceva}).
We introduce trilinear coordinates associated with such triangles
and develop several geometric constructions.
In particular, for any proper M\"obius triangle we define the incenter,
excenters, and orthocenter. An interactive demo containing most of our results is available at
\url{https://github.com/livinsky-research/Moebius}
\end{abstract}

\begin{figure}[h]
\[
\begin{tikzpicture}
\tikzstyle{vertex}=[circle, fill, inner sep=1.5pt]
\draw [thick](6.000000,2.000000) arc (196.928605:115.666103:5.500987);
\draw [thick](13.260000,3.420000) arc (-124.794809:-154.333897:13.244863);
\node [vertex] at (13.260000,3.420000) {};
\node [vertex] at (8.880000,8.560000) {};
\node [vertex] at (16.45,1.80000) {};
\draw [thick](6.000000,2.000000) arc (126.928605:75.205191:8.479550);
\node [vertex] at (6.000000,2.000000) {};
\draw (6.000000,2.000000) arc (167.882244:0:5.326801);
\draw (8.880000,8.560000) arc (133.928993:333:5.094527);
\draw (13.260000,3.420000) arc (69.074771:99:16.926158);
\draw (13.260000,3.420000) arc (69.074771:55:16.926158);
\node [vertex] at (7.332268,4.535918) {};
\end{tikzpicture}
\]
\caption{Ceva's Theorem for a M\"obius triangle gives necessary and sufficient conditions for three cevians to belong to one pencil.}
\label{fig:ceva}
\end{figure}
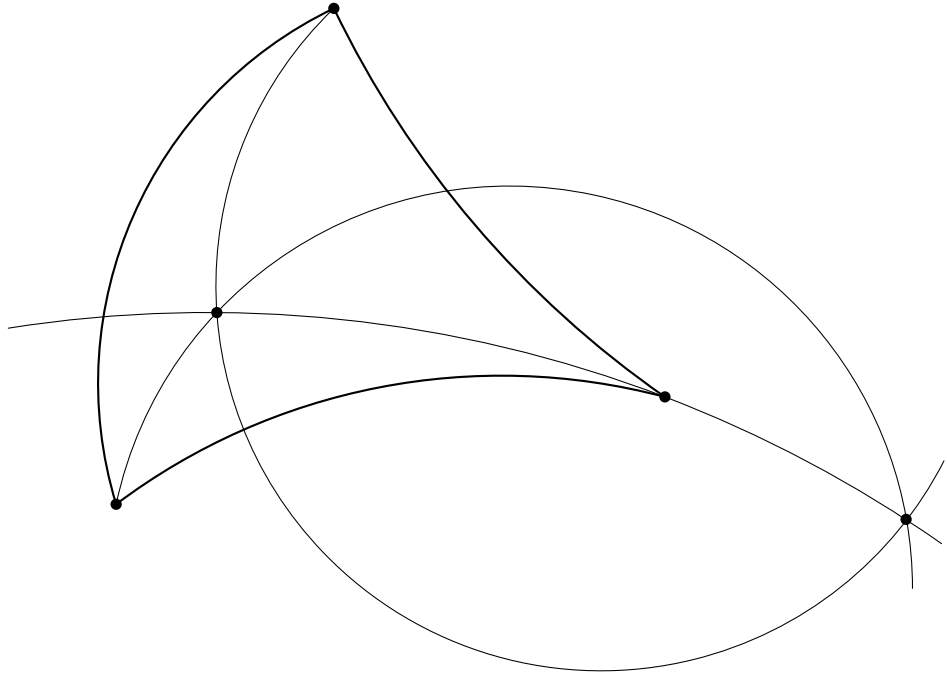

\section{The M\"obius plane}
The main purpose of this paper is to show that M\"obius geometry is no less rich than classical geometries such as affine or hyperbolic. As we shall see, the results obtained here naturally yield corresponding statements
for spherical, hyperbolic, and Euclidean geometries, illustrating the unifying role of M\"obius geometry.

We begin by recalling the basic properties of the M\"obius plane.
The M\"obius plane $\M$ is an incidence structure of \textit{points} and \textit{cycles}. 
In its classical model we identify $\M$ with the Riemann sphere $\C\cup\{\infty\}$ (or, equivalently, with the complex projective line $\C P^1$). In this representation the cycles are all circles and all straight lines with $\infty$ attached. For any three points there exists a unique cycle containing them. Straight lines are exactly the cycles containing $\infty$. 

Any cycle $\ca$ can be defined by an equation of the form
$$
k z\oline{z} + l z + m\oline{z} + n=0,
$$
where the matrix $M=\left(\begin{smallmatrix}k & l\\ m & n\end{smallmatrix}\right)\in\Herm(2)$ is Hermitian, non-degenerate, and indefinite \cite{Sch}. In particular, $k$ and $n$ are real, $l=\oline{m}$, and $\det M<0$. Two matrices define the same cycle whenever they differ by a real scalar factor. A cycle is a straight line iff $k=0$. When $k\neq 0$ the cycle is a circle with center $O=-\frac{m}{k}$ and radius $r$ satisfying $r^2=\frac{|m|^2}{k^2}-\frac{n}{k}=-\frac{1}{k^2}\det M$.

The M\"obius plane has the isometry group $\Iso\M$ that maps cycles to cycles and preserves incidence. Two sets $X, Y\subseteq\M$ are \textit{congruent} if there exists an isometry $\sigma\in\Iso\M$ such that $\sigma(X)=Y$.
The group $\Iso\M$ is generated by all inversions $\sigma_\ca$ about all cycles $\ca$. In the case of line the corresponding inversion is the reflection across the line.
The isometry group has two connectedness components. The component of the identity is the orientation-preserving subgroup $\Iso^+\M$. It is a subgroup of index $2$ and consists of all M\"obius transformations
$$
z\mapsto\frac{az+b}{cz+d}
$$
and is naturally isomorphic to the group $\PSL(2;\C)$.
The remaining isometries are maps of the form
$$
z\mapsto\frac{a\oline{z}+b}{c\oline{z}+d}.
$$
In particular, all rotations, translations, homotheties, and similarities are M\"obius isometries. The only isometries that preserve $\infty$ are of the form $z\mapsto az+b$ or $z\mapsto a\oline{z}+b$, for some $a,b\in\C$, $a\neq 0$. These are exactly the Euclidean similarities. The full group $\Iso\M$ is generated by $\Iso^+\M$ and any inversion $\sigma_\ca$, or by all similarities and any inversion $\sigma_\ca$ provided that $\ca$ is a circle.

All the isometry transformations are conformal. Therefore, angles are well-defined in M\"obius geometry and agree with the angles in the Euclidean representation. On the other hand, there is no distance, no area, and no in-between relation for triples of points; moreover, every cycle has no center.

The standard constant-curvature geometries (spherical, Euclidean, and hyperbolic) are naturally embedded into M\"obius geometry, and every theorem in M\"obius geometry becomes some result in these geometries involving lines and circles. 
\begin{itemize}
\item for spherical geometry we identify $\M$ with the sphere $S^2$ via stereographic projection. The cycles are then all straight lines and all circles.
All are obtained as planar sections of the sphere. The planes corresponding to the
straight lines are the ones passing through the center of $S^2$. After projecting   onto $\C$ the straight lines become the lines passing through $0$ and the circles relative to which the power of $0$ is $-1$. In the matrix language these cycles correspond to the matrices $M$ with an additional property $k+n=\trr M=0$.

\item for Euclidean geometry we consider $\C=\M\setminus\{\infty\} $ and once again the cycles are all straight lines and all circles. The straight lines are identified by the condition $k=0$. 

\item for hyperbolic geometry we consider the upper half-plane 
$\hh\subset\C$ as the Poincar\'e model. The straight lines are the half-circles orthogonal to $\R$ and vertical rays. These are exactly the cycles defined by real matrices $M$. The remaining 
cycles that intersect $\hh$ cover all hyperbolic circles, horocycles, and equidistant curves depending on the number of intersection points with $\oline{\R}=\R\cup\{\infty\}$ (this is also the real projective line $\R P^1$).
\end{itemize}

In all our proofs we are not going beyond complex numbers, trigonometry, basic linear algebra, and quadratic equations.

\section{Points and cycles}
The following property of M\"obius transformations is well-known from Complex Analysis.
\begin{prop}
For any two triples $A, B, C$ and $A^\pr, B^\pr, C^\pr$ of distinct points on $\M$ there exists a unique orientation-preserving 
isometry $\sigma\in\Iso^+ \M$ such that $\sigma(A)=A^\pr$, $\sigma(B)=B^\pr$, $\sigma(C)=C^\pr$.
\end{prop}
It means that any three distinct points on the M\"obius plane are in general position. Hence, not only there are no two-point invariants such as distance, but no three-point invariants either. Only for four points we have some special cases and nontrivial invariants.

For any two points there are infinitely many cycles passing through them. Two points lying on a cycle split it into two cycle segments.
 Let  $A, B, C, D$ be four distinct points
lying on some cycle. We say that $A,B$ \textit{split} $C,D$ if $C$ and $D$ belong to different segments formed by $A$ and $B$. For any $A, B, C, D$ we have either $A,B$ split $C,D$, or $A,C$ split $B, D$, or $A, D$ split $B, C$ (Figure~\ref{fig:splitting}). 

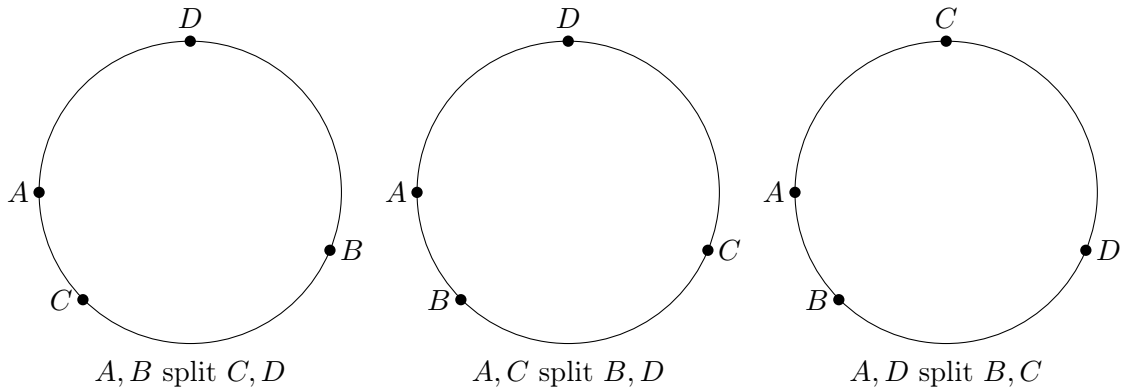
\begin{figure}[h]
\[
\begin{tikzpicture}
\tikzstyle{vertex}=[circle, fill, inner sep=1.5pt]
\draw (2,2) circle (2cm);
\node [vertex] at (0,2) {};
\node [left] at (0,2) {$A$};
\node [vertex] at (0.58,0.58) {};
\node [left] at (0.58,0.58) {$C$};
\node [vertex] at (2,4) {};
\node at (2,4.3) {$D$};
\node [vertex] at (3.848,1.234) {};
\node [right] at (3.848,1.234) {$B$};
\node at (2,-0.4) {$A,B$ split $C,D$};
\draw (7,2) circle (2cm);
\node [vertex] at (5,2) {};
\node [left] at (5,2) {$A$};
\node [vertex] at (5.58,0.58) {};
\node [left] at (5.58,0.58) {$B$};
\node [vertex] at (7,4) {};
\node at (7,4.3) {$D$};
\node [vertex] at (8.848,1.234) {};
\node [right] at (8.848,1.234) {$C$};
\node at (7,-0.4) {$A,C$ split $B,D$};
\draw (12,2) circle (2cm);
\node [vertex] at (10,2) {};
\node [left] at (10,2) {$A$};
\node [vertex] at (10.58,0.58) {};
\node [left] at (10.58,0.58) {$B$};
\node [vertex] at (12,4) {};
\node at (12,4.3) {$C$};
\node [vertex] at (13.848,1.234) {};
\node [right] at (13.848,1.234) {$D$};
\node at (12,-0.4) {$A,D$ split $B,C$};
\end{tikzpicture}
\]
\caption{Relative position of four points on a cycle.}
\label{fig:splitting}
\end{figure}

Furthermore, the cross-ratio
$$
(A,B;C,D) = \frac{C-A}{C-B}\frac{D-B}{D-A}
$$
is real and is invariant under M\"obius isometries. However, we will not need it in this paper.


Now consider the relative position of cycles on the M\"obius plane.
Every cycle splits the plane into two \textit{half-planes}. Two cycles $\ca, \cb$ are either disjoint, intersect in two points, or intersect in one point. In the last case we say that $\ca$ and $\cb$ are \textit{tangent}. For example, two parallel lines are tangent in $\infty$.

For three pairwise disjoint cycles $\ca$, $\cb$, $\cc$ we naturally have
the in-between relation. We say that $\ca$ \textit{is between} $\cb$ and $\cc$, or $\ca$ \textit{splits $\cb$ and $\cc$} if $\cb$ and $\cc$ lie in different half-planes bounded by $\ca$. Clearly, for any three pairwise disjoint cycles either one of them or none is lying between two others (Figure~\ref{fig:between}).

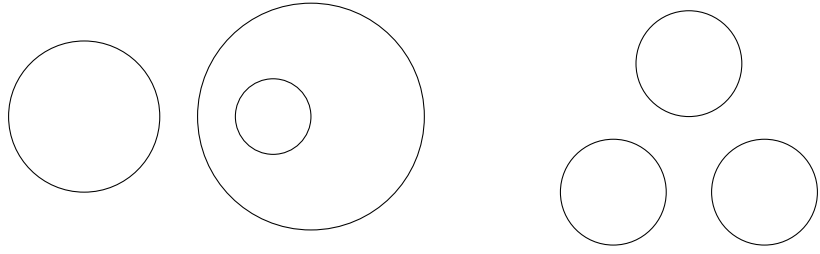
\begin{figure}[h]
\[
\begin{tikzpicture}
\draw (4,2) circle (1.5cm);
\draw (3.5,2) circle (0.5cm);
\draw (1,2) circle (1cm);
\draw (9, 2.7) circle(0.7cm);
\draw (8, 1) circle(0.7cm);
\draw (10, 1) circle(0.7cm);
\end{tikzpicture}
\]
\caption{Relative position of three pairwise disjoint cycles.}
\label{fig:between}
\end{figure}

Two disjoint cycles split the plane into two half-planes and an \textit{annulus}. For example, a ring between two concentric circles is an annulus.
Two tangent cycles split the plane into two half-planes and a new region -- a \textit{monogon} -- a domain bounded by two cycles and one point -- its vertex. For example, a strip between two parallel lines is a monogon with $\infty$ being its vertex.

\begin{prop}
Any two monogons are congruent.
\end{prop}
\begin{proof}
Let $\ca, \cb$ be two tangent cycles. Without loss of generality assume that
their common point is $\infty$. Then $\ca$ and $\cb$ are simply a pair of parallel lines, and the monogon is a strip between them. By applying rotations, translations, and scaling we can superimpose exactly one strip onto another.
\end{proof}
On the other hand, two given annuli are not congruent in general.
For two intersecting cycles the intersection angles in the two intersection points are equal; therefore, we can introduce the angle between the cycles in this case. In particular, we naturally define \textit{orthogonal} cycles and use the standard notation $\ca\perp\cb$.

A cycle $\cc$ is a \textit{midcycle} for two cycles $\ca$ and $\cb$ if $\sigma_\cc(\ca)=\cb$, where $\sigma_\cc\in\Iso\M$ is the inversion about $\cc$.

\begin{prop}[\cite{J}] Let $\ca,\cb$ be two distinct cycles. Then
\begin{enumerate}
\item[$(i)$] if $\ca$ and $\cb$ do not intersect then there exists a unique midcycle $\cc$. It is represented by a straight line iff $\ca$ and $\cb$ have equal radii.
\item[$(ii)$] if $\ca$ and $\cb$ are tangent then there exists a unique midcycle. It is tangent to both $\ca$ and $\cb$ in the same point.
\item[$(iii)$] if $\ca$ and $\cb$ intersect then there are two midcycles. Both are passing through the intersection points of $\ca$ and $\cb$ and are mutually orthogonal.
\end{enumerate}
\end{prop}

\section{Orientation}
The M\"obius plane $\M$ is an orientable manifold and we can naturally introduce orientation. We consider oriented cycles and angles. In the Euclidean representation we assume the counter-clockwise rotation is positive.
For every oriented cycle each half-plane bounded by it is uniquely labeled to be its \textit{left} of \textit{right} depending on which side it appears while moving along the cycle following its orientation. Thus, for positively-oriented circles the inner half-plane (i.e. the one not containing $\infty$) is its left, and for negatively oriented circles it is right. We use notation $\ca\cong\cb$ if $\ca$ and $\cb$ agree as sets regardless of orientation.

For two oriented cycles $\ca,\cb$ and their intersection point $C$ we define the \textit{oriented angle} 
$\alpha=\angle(\ca,\cb)$ as the Euclidean oriented angle $\alpha$ between the unit tangent vectors in $C$. We assume $\alpha\in(\pi,\pi]$. Note, that the intersection point $C$ must be explicitly specified as the angle depends on it. However, in all our results the expressions involving angles will not depend on which intersection point was chosen in the end. 

In the matrix language we can distinguish the orientation by considering matrices up to a positive factor. For every cycle $\ca$ we can normalize its matrix $M=\left(\begin{smallmatrix}k & l\\ m & n\end{smallmatrix}\right)$ by imposing the condition $\det M=-1$. Then every oriented cycle corresponds to a unique normalized matrix if we assume that the left half-plane is defined by the inequality
$$
kz\oline{z}+lz+m\oline{z} + n<0.
$$
We use notation $\ca\sim M$ when a matrix $M$ defines a cycle $\ca$, and $\ca\approx M$ when a normalized matrix $M$ defines an oriented cycle $\ca$.
Therefore, $k>0$ defines positively oriented circles, and $k<0$ is for negatively oriented. Moreover, for an oriented circle we define its \textit{signed radius} by the formula $r=\frac{1}{k}\sqrt{-\det M}$; when $M$ is normalized then we have simply $r=\frac{1}{k}$.

Let $\ca\approx M$. The defining equation of $\ca$ can be written in the matrix form as
$\left(
\begin{smallmatrix}
z  \\
1 \\
\end{smallmatrix}
\right)^\top
M
\left(
\begin{smallmatrix}
\oline{z}  \\
1 \\
\end{smallmatrix}
\right)=0$.
Let $\sigma\in\Iso^+\M$ be a M\"obius transformation with the matrix 
$g=\left(\begin{smallmatrix}a & b\\ c & d\end{smallmatrix}\right)$, $\det g=1$. Then it is easy to see that $\sigma(\ca)\approx g^{-\top}M\oline{g}^{-1}$.
The complex conjugation $z\mapsto\oline{z}$ (this is simply the reflection across $\R$) acts by transposition on $M$.

Therefore, we have an action of the group $G=\Iso\M$ on the space
$\Herm(2)$ of all $2\times 2$ Hermitian matrices that agrees with the action on the oriented cycles. Moreover, consider the bilinear functional
$$
\la M,N\ra = \frac{1}{2}\left(\trr MN-\trr M\trr N\right)
$$
on $\Herm(2)$. It defines a real pseudoeuclidean inner product of signature $(3,1)$ -- a Minkowski metric. Furthermore, for $M=\left(\begin{smallmatrix}k & l\\ m & n\end{smallmatrix}\right)$  we have that
$$
\la M, M\ra = \frac{1}{2}(\trr M^2-(\trr M)^2)=
\frac{1}{2}\big(k^2+2lm+n^2-(k+n)^2\big) = lm-kn=-\det M.
$$
Clearly, $\det M$ is invariant with respect to $G$. Therefore, it follows via polarization that $G$ preserves the inner product.

Let $\ca,\cb\approx M, N$ respectively. The value
$$
\la\ca,\cb\ra = \la M,N\ra
$$
is, therefore, a M\"obius invariant, i.e. $\la \sigma(\ca),\sigma(\cb)\ra= \la\ca,\cb\ra$ for all $\sigma\in\Iso\M$.

\begin{prop}\label{cs}
Let $\ca,\cb$ be two oriented cycles represented by circles of signed radii $r_a$, $r_b$. Let $d$ be the distance between their centers. Then
$$
\la\ca,\cb\ra=\frac{r_a^2 + r_b^2 - d^2}{2r_a r_b}.
$$
\end{prop}
\begin{proof}
Let $\ca\approx M=\left(\begin{smallmatrix}k_a & l_a\\ m_a & n_a\end{smallmatrix}\right)$, 
$\cb\approx N=\left(\begin{smallmatrix}k_b & l_b\\ m_b & n_b\end{smallmatrix}\right)$. Then
\begin{align*}
\trr MN &= k_ak_b + l_am_b+l_bm_a + n_an_b,\\
\trr M\trr N &= (k_a+n_a)(k_b+n_b) = k_ak_b + k_an_b+k_bn_a+n_an_b.
\end{align*}
Therefore,
$$
\la\ca,\cb\ra  = \frac{1}{2}(-k_an_b-k_bn_a+l_am_b+l_bm_a).
$$
On the other hand, $\ca$ and $\cb$ have centers $O_a=-\frac{m_a}{k_a}$, $O_b=-\frac{m_b}{k_b}$ and signed radii $r_a=\frac{1}{k_a}$, $r_b=\frac{1}{k_a}$. Thus,
\begin{gather*}
d^2 = |O_a-O_b|^2=\left|\frac{m_a}{k_a}-\frac{m_b}{k_b}\right|^2 = 
\left(\frac{l_a}{k_a}-\frac{l_b}{k_b}\right)
\left(\frac{m_a}{k_a}-\frac{m_b}{k_b}\right)=\\
=\frac{l_am_a}{k_a^2}  -\frac{l_am_b+l_bm_a}{k_ak_b}+ 
\frac{l_bm_b}{k_b^2}=
\frac{1+k_an_a}{k_a^2}  -\frac{l_am_b+l_bm_a}{k_ak_b}+ 
\frac{1+k_bn_b}{k_b^2}=\\
=\frac{1}{k_a^2}+\frac{1}{k_b^2}+
\frac{k_an_b+k_bn_a-l_am_b-l_bm_a}{k_ak_b}
=r_a^2+r_b^2-2r_ar_b\la\ca,\cb\ra.
\end{gather*}
\end{proof}
When $\cb$ is a line we can show that $\la\ca,\cb\ra = \frac{d}{r_a}$ where $d$ is the signed distance from the center of $\ca$ to $\cb$; positive when the center is in the left half-plane of $\cb$ and negative otherwise. 

We call $\la\ca,\cb\ra$ the \textit{M\"obius cosine} of $\ca$ and $\cb$.
Trivially, we always have $\la\ca,\cb\ra=\la\cb,\ca\ra$ and when $\ca$ and $\cb$ intersect $\la\ca,\cb\ra=\cos\angle(\ca,\cb)$.

A real-valued function of an oriented cycle argument is said to be \textit{even} or \textit{odd} depending on whether it changes the sign when the argument's orientation is reversed. For example, the M\"obius cosine is symmetric and an odd function in both arguments.

\section{Pencils}
Consider the spaces $\widetilde{\CC}$, $\CC$ of all oriented and  non-oriented cycles on $\M$ respectively. A topology is naturally introduced that turns both into three-dimensional manifolds.
\begin{prop}
The space $\widetilde{\CC}$ is homeomorphic to the product $S^2\times(-1,1)$.
\end{prop}
\begin{proof}
Identify $\M$ with $S^2$ -- the unit sphere centered at $0$. All cycles then are intersections of $S^2$ with the planes that lie at distance $d<1$ from the origin.
For every cycle $\ca\in\widetilde{\CC}$ consider the line $\ell$ passing through the origin and orthogonal to the plane containing $\ca$.
It intersects the sphere in two antipodal points.
The orientation of $\ca$ naturally defines the orientation on $\ell$. Moreover, a coordinate $t$ is naturally introduced on $\ell$ such that the origin corresponds to $t=0$ and the intersection points with $S^2$ to $t=\pm1$; we also assume that orientation on $\ell$ agrees with the standard orientation on $\R$.
Then for every cycle $\ca$ we select the right intersection point $x$ (the one corresponding to $t=1$) and the coordinate $t$ of the intersection point of $\ell$ with the plane containing $\ca$.
Thus, for each $\ca$ we have constructed a pair $(x,t)$, where $x\in S^2$ and $t\in(-1,1)$. Clearly, this map is a homeomorphism.
\end{proof}

\begin{corollary}
The space $\CC$ is homeomorphic to the quotient of $\widetilde{\CC}$ by the equivalence relation $(x,t)\sim(-x,-t)$ i.e. to the elliptic de Sitter 3-space.
The orientation-forgetful map $p:\widetilde{\CC}\rightarrow\CC$ 
is a universal two-fold covering of $\CC$.
\end{corollary}

The space of cycles $\CC$ has a geometry of its own.
A \textit{pencil} is a one-dimensional family of cycles
that naturally appears in M\"obius geometry. Let $\ca,\cb\approx M,N$ be two distinct oriented  cycles. The pencil 
$\CC(\ca,\cb) = \{\cc\sim xM+yN|x,y\in\R\}$
 is a family containing all non-oriented cycles that are real linear combinations of $\ca$ and $\cb$. If $\ca\not\cong\cb$ then $M$ and $N$ are not proportional and generate a subspace $W$ of dimension $2$ in $\Herm(2)$. The pencil $\CC(\ca,\cb)$ is a collection of cycles represented by 
matrices contained in $W$. Pencils are lines in $\CC$. Three cycles $\ca,\cb,\cc$ are \textit{collinear} when they belong to one pencil. Three pencils are \textit{concurrent} when they share one common cycle. The action of $\Iso\M$ on $\Herm(2)$ is linear; therefore, it maps pencils to pencils. 

Consider the properties and types of pencils. The next result is trivial.

\begin{prop}
Any two pencils either do not intersect, or share a unique common cycle.
\end{prop}

Pencils are naturally distinguished by their \textit{type}. Let $\gP=\CC(\ca,\cb)$. If $\ca$ and $\cb$ intersect in two points $A,B$ then these points uniquely define the pencil $\gP$. Every other cycle passing through $A$ and $B$ belongs to $\gP$. In this case $\gP$ is \textit{elliptic}.
If $\ca$ and $\cb$ are tangent in a point $A$, then all cycles in $\gP$ are tangent to each other in $A$. This pencil is \textit{parabolic}.
Finally, if $\ca$ and $\cb$ do not intersect, then we have a \textit{hyperbolic} pencil. It is uniquely defined by a pair of asymptotic points that do not belong to any cycle from $\gP$. The common intersection points or the asymptotic ones are the \textit{distinguished points} of a pencil. When considering Euclidean representation, we call a pencil \textit{standard} if 
$\infty$ is not one of its distinguished points (Figure~\ref{fig:standard}), and \textit{special} otherwise. A special pencil is either a collection of straight lines passing through one point, a collection of parallel lines, or a collection of concentric circles (Figure~\ref{fig:special}). 

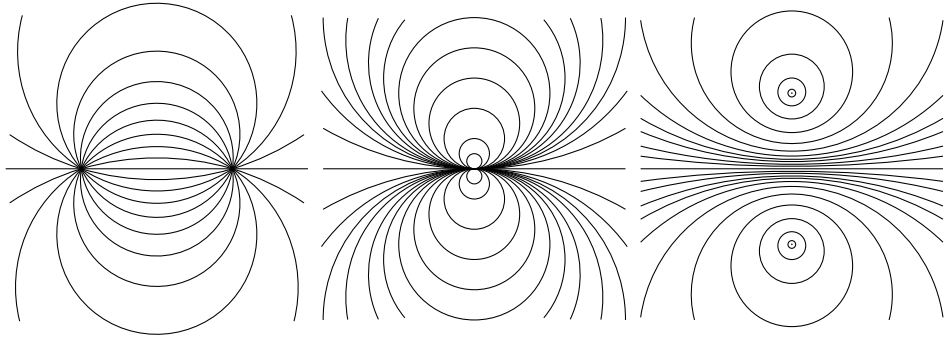
\begin{figure}[h]
\[
\begin{tikzpicture}
\tikzstyle{vertex}=[circle, fill, inner sep=1.5pt]
]
\draw(0.8, 2.0) -- (4.8, 2.0);

\draw(1.8, 2.0) arc (106.36363636363636:73.63636363636364:3.549465532884223);
\draw(1.8, 2.0) arc (253.63636363636363:286.3636363636364:3.549465532884223);
\draw(1.8, 2.0) arc (122.72727272727272:57.27272727272727:1.8496568659138088);
\draw(1.8, 2.0) arc (237.27272727272728:302.72727272727275:1.8496568659138088);
\draw(1.8, 2.0) arc (139.0909090909091:40.909090909090914:1.323189630446618);
\draw(1.8, 2.0) arc (220.9090909090909:319.09090909090907:1.323189630446618);
\draw(1.8, 2.0) arc (155.45454545454544:24.545454545454547:1.0993456750718866);
\draw(1.8, 2.0) arc (204.54545454545456:335.45454545454544:1.0993456750718866);
\draw(1.8, 2.0) arc (171.8181818181818:8.181818181818173:1.0102832265380364);
\draw(1.8, 2.0) arc (188.1818181818182:351.8181818181818:1.0102832265380364);
\draw(1.8, 2.0) arc (188.1818181818182:-8.181818181818173:1.0102832265380361);
\draw(1.8, 2.0) arc (171.8181818181818:368.1818181818182:1.0102832265380361);
\draw(1.8, 2.0) arc (204.54545454545456:-24.545454545454547:1.0993456750718864);
\draw(1.8, 2.0) arc (155.45454545454544:384.54545454545456:1.0993456750718864);
\draw(1.8, 2.0) arc (220.9090909090909:-40.90909090909091:1.323189630446618);
\draw(1.8, 2.0) arc (139.0909090909091:400.9090909090909:1.323189630446618);
\draw(1.8, 2.0) arc (237:165:1.8496568659138082);
\draw(1.8, 2.0) arc (122:194:1.8496568659138082);
\draw(3.8, 2.0) arc (303:375:1.8496568659138082);
\draw(3.8, 2.0) arc (58:-14:1.8496568659138082);
\draw(1.8, 2.0) arc (253:236:3.5494655328842235);
\draw(1.8, 2.0) arc (107:124:3.5494655328842235);
\draw(3.8, 2.0) arc (287:304:3.5494655328842235);
\draw(3.8, 2.0) arc (73:56:3.5494655328842235);

\draw(9, 2.0) -- (5, 2.0);
\draw(7, 2.1) circle (0.1);
\draw(7, 2.2) circle (0.2);
\draw(7, 2.4) circle (0.4);
\draw(7, 2.6) circle (0.6);
\draw(7, 2.8) circle (0.8);
\draw(7, 3) circle (1);
\draw(7, 1) circle (1);
\draw(7, 1.2) circle (0.8);
\draw(7, 1.4) circle (0.6);
\draw(7, 1.6) circle (0.4);
\draw(7, 1.8) circle (0.2);
\draw(7, 1.9) circle (0.1);
\draw(7, 2) arc (90:180: 2);
\draw(7, 2) arc (90:0: 2);
\draw(7, 2) arc (270:180: 2);
\draw(7, 2) arc (-90:0: 2);
\draw(7, 2) arc (90:220: 1.2);
\draw(7, 2) arc (90:205: 1.4);
\draw(7, 2) arc (90:190: 1.7);
\draw(7, 2) arc (90:144: 2.5);
\draw(7, 2) arc (90:120: 4.0);

\draw(7, 2) arc (90:-40: 1.2);
\draw(7, 2) arc (90:-25: 1.4);
\draw(7, 2) arc (90:-10: 1.7);
\draw(7, 2) arc (90:36: 2.5);
\draw(7, 2) arc (90:60: 4.0);

\draw(7, 2) arc (270:140: 1.2);
\draw(7, 2) arc (270:155: 1.4);
\draw(7, 2) arc (270:170: 1.7);
\draw(7, 2) arc (270:216: 2.5);
\draw(7, 2) arc (270:240: 4.0);

\draw(7, 2) arc (-90:40: 1.2);
\draw(7, 2) arc (-90:25: 1.4);
\draw(7, 2) arc (-90:10: 1.7);
\draw(7, 2) arc (-90:-36: 2.5);
\draw(7, 2) arc (-90:-60: 4.0);

\draw(9.2, 2.0) -- (13.2, 2.0);

\draw(11.2, 3.0000049999875) circle (0.00316227766016838);
\draw(11.2, 0.9999950000125) circle (0.00316227766016838);
\draw(11.2, 3.0013009053514006) circle (0.051024533878662764);
\draw(11.2, 0.9986990946485992) circle (0.051024533878662764);
\draw(11.2, 3.0168295978569843) circle (0.1842347173526109);
\draw(11.2, 0.9831704021430157) circle (0.1842347173526109);
\draw(11.2, 3.087943860484949) circle (0.4285111942142172);
\draw(11.2, 0.9120561395150508) circle (0.4285111942142172);
\draw(11.2, 3.28344374575134) circle (0.804504722489701);
\draw(11.2, 0.7165562542486599) circle (0.804504722489701);

\draw(12.486862135475985, 4.0) arc (14.633099829693208: -194.6330998296932:1.3300029238964377);
\draw(9.913137864524014, -6.661338147750939e-16) arc (194.6330998296932: -14.633099829693208:1.3300029238964377);
\draw(13.2, 3.965034860622859) arc (351.7573392407591: 188.2426607592409:2.0208761370574218);
\draw(13.2, 0.03496513937714213) arc (8.242660759240906: 171.7573392407591:2.0208761370574218);
\draw(13.2, 2.9712511390515473) arc (313.7616915644875: 226.23830843551252:2.8915922783899397);
\draw(13.2, 1.0287488609484527) arc (46.23830843551251: 133.76169156448748:2.8915922783899397);
\draw(13.2, 2.6673212101310555) arc (300.3725589028263: 239.62744109717372:3.955534914636118);
\draw(13.199999999999998, 1.3326787898689454) arc (59.62744109717371: 120.37255890282628:3.955534914636118);
\draw(13.2, 2.492739388217098) arc (292.50470327023515: 247.49529672976482:5.225216360038723);
\draw(13.2, 1.507260611782903) arc (67.49529672976483: 112.50470327023517:5.225216360038723);

\draw(13.200000000000003, 2.299848338432181) arc (283.726896361168: 256.273103638832:8.428356225816907);
\draw(13.200000000000001, 1.7001516615678192) arc (76.27310363883204: 103.72689636116796:8.428356225816907);

\draw(13.200000000000005, 2.1666369180765717) arc (277.6352101870891: 262.3647898129109:15.052816502866621);
\draw(13.200000000000001, 1.8333630819234266) arc (82.3647898129109: 97.6352101870891:15.052816502866621);

\end{tikzpicture}
\]
\caption{Standard pencils: elliptic, parabolic, hyperbolic.}
\label{fig:standard}
\end{figure}

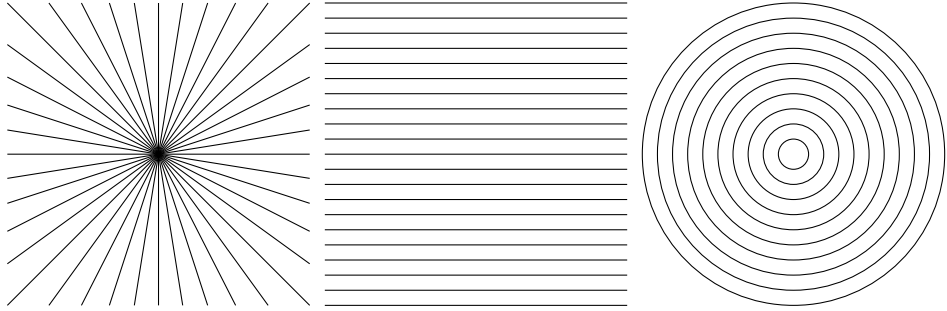
\begin{figure}[h]
\[
\begin{tikzpicture}
\tikzstyle{vertex}=[circle, fill, inner sep=1.5pt]

\draw(4.80, 2.00) -- (0.80, 2.00);
\draw(4.80, 2.32) -- (0.80, 1.68);
\draw(4.80, 2.65) -- (0.80, 1.35);
\draw(4.80, 3.02) -- (0.80, 0.98);
\draw(4.80, 3.45) -- (0.80, 0.55);
\draw(4.80, 4.00) -- (0.80, 0.00);
\draw(4.25, 4.00) -- (1.35, 0.00);
\draw(3.82, 4.00) -- (1.78, 0.00);
\draw(3.45, 4.00) -- (2.15, 0.00);
\draw(3.12, 4.00) -- (2.48, 0.00);
\draw(2.80, 4.00) -- (2.80, 0.00);
\draw(2.48, 4.00) -- (3.12, 0.00);
\draw(2.15, 4.00) -- (3.45, 0.00);
\draw(1.78, 4.00) -- (3.82, 0.00);
\draw(1.35, 4.00) -- (4.25, 0.00);
\draw(0.80, 4.00) -- (4.80, 0.00);
\draw(0.80, 3.45) -- (4.80, 0.55);
\draw(0.80, 3.02) -- (4.80, 0.98);
\draw(0.80, 2.65) -- (4.80, 1.35);
\draw(0.80, 2.32) -- (4.80, 1.68);

\draw(9, 4) -- (5, 4);
\draw(9, 3.8) -- (5, 3.8);
\draw(9, 3.6) -- (5, 3.6);
\draw(9, 3.4) -- (5, 3.4);
\draw(9, 3.2) -- (5, 3.2);
\draw(9, 3.0) -- (5, 3.0);
\draw(9, 2.8) -- (5, 2.8);
\draw(9, 2.6) -- (5, 2.6);
\draw(9, 2.4) -- (5, 2.4);
\draw(9, 2.2) -- (5, 2.2);
\draw(9, 2.0) -- (5, 2.0);
\draw(9, 1.8) -- (5, 1.8);
\draw(9, 1.6) -- (5, 1.6);
\draw(9, 1.4) -- (5, 1.4);
\draw(9, 1.2) -- (5, 1.2);
\draw(9, 1.0) -- (5, 1.0);
\draw(9, 0.8) -- (5, 0.8);
\draw(9, 0.6) -- (5, 0.6);
\draw(9, 0.4) -- (5, 0.4);
\draw(9, 0.2) -- (5, 0.2);
\draw(9, 0.0) -- (5, 0.0);

\draw (11.2,2) circle (2cm);
\draw (11.2,2) circle (1.8cm);
\draw (11.2,2) circle (1.6cm);
\draw (11.2,2) circle (1.4cm);
\draw (11.2,2) circle (1.2cm);
\draw (11.2,2) circle (1.0cm);
\draw (11.2,2) circle (0.8cm);
\draw (11.2,2) circle (0.6cm);
\draw (11.2,2) circle (0.4cm);
\draw (11.2,2) circle (0.2cm);
\end{tikzpicture}
\]
\caption{Special pencils: elliptic, parabolic, hyperbolic.}
\label{fig:special}
\end{figure}

\begin{prop}[\cite{J}]
Let $\ca,\cb$ be two cycles. Then the collection of cycles 
$\gQ=\{\cc\in\CC|\cc\perp\ca,\cc\perp\cb\}$ is a pencil. Moreover, every cycle in $\gP=\CC(\ca,\cb)$ is orthogonal to every cycle in $\gQ$.
\end{prop}
Such pencils are \textit{orthogonal}. We use notation $\gQ=\gP^\perp$ in this case. Clearly, we have $\gP=\gP^{\perp\perp}$ for every pencil $\gP$.

\begin{prop}
The pencils $\gP$ and $\gP^\perp$ share the same distinguished points. If $\gP$ is elliptic then $\gP^\perp$ is hyperbolic and vice versa; $\gP$ is parabolic iff $\gP^\perp$ is so too.
\end{prop}
 
\begin{prop}[\cite{J}]
Every standard pencil contains exactly one straight line, and it is a radical axis of any two circles in it. In a special pencil all cycles are either straight lines or circles.
\end{prop}


\begin{prop}
Let $\ca,\cb,\cc$ be collinear. Then
$$
\la\ca,\cb\ra^2 + \la\cb,\cc \ra^2 + \la\ca,\cc \ra^2 -
2 \la\ca,\cb\ra\la\cb,\cc\ra\la\ca,\cc\ra=1. 
$$
\end{prop}
\begin{proof}
If $\ca,\cb,\cc$ are collinear then the Gram matrix representing them is degenerate. We have then
$$
0=\left|
\begin{matrix}
1 & \la\ca,\cb\ra & \la\ca,\cc\ra \\
\la\ca,\cb\ra & 1 & \la\cb,\cc\ra \\
\la\ca,\cc\ra & \la\cb,\cc\ra & 1\\
\end{matrix}
\right|=1+2\la\ca,\cb\ra\la\cb,\cc\ra\la\ca,\cc\ra
-\la\ca,\cb\ra^2 - \la\cb,\cc \ra^2 - \la\ca,\cc \ra^2.
\qedhere
$$
\end{proof}
Note, that the opposite is not true. If $\ca,\cb,\cc$ are not collinear, but generate a degenerate subspace in $\Herm(2)$ the identity will also be true.

\medskip
For the standard geometries we consider \textit{linear} pencils, i.e. those that contain straight lines of that geometry only. We can see that a pencil is linear iff it contains at least two straight lines. Let $\LL$ be the space of linear pencils. We will see that in all geometries $\LL$ is homeomorphic to $\R P^2$.

\begin{itemize}
\item in spherical geometry every linear pencil $\gP$ is elliptic, and is uniquely defined by a pair of antipodal points; it consists of all meridians passing through these points when viewed as poles. The orthogonal pencil $\gP^\perp$ is hyperbolic and contains all parallels related to the same poles. 
After gluing antipodal points we obtain a homeomorphism of $\LL$ and $\R P^2$.

\item in Euclidean geometry linear pencils are either elliptic or parabolic. All of them are special. The elliptic ones are in bijection with the points. And parabolic are the pencils consisting of parallel lines. These are in bijection with the line at infinity attached to the plane. Together we obtain the projective plane again.

\item in hyperbolic geometry all types are present. An elliptic pencil $\gP$ 
consists of lines passing through some point $O\in\hh$; the orthogonal pencil $\gP^\perp$ contains all circles centered in $O$. The space of elliptic pencils is, therefore, homeomorphic to $\hh$. Parabolic pencils are identified with the points at infinity. Orthogonal pencils contain horocycles sharing a common center. Finally, every hyperbolic pencil $\gP$ is uniquely defined by a pair of points at infinity. A line $\cl$ joining them is the unique line belonging to $\gP^\perp$; the rest of $\gP^\perp$ are the equidistant curves for $\cl$. Thus, all hyperbolic pencils are in one-to-one correspondence with the lines in hyperbolic geometry. The line at infinity is homeomorphic to the circle $S^1$. Therefore, the space of lines is homeomorphic to
the quotient space  $(S^1\times S^1\setminus\{(x,x)|x\in S^1\})/\sim$,
where $(x,y)\sim (y,x)$. That is, we take a torus, cut it along the diagonal embedding, and take one of the two equal parts.
We obtain a triangle in which two sides are identified (Figure~\ref{fig:gluing}). After gluing back the elliptic pencils we obtain $\R P^2$ once again.

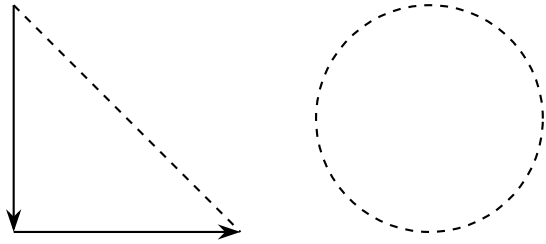
\begin{figure}[h]
\[
\begin{tikzpicture}
\draw [thick][-{Stealth[length=3mm, width=2mm]}](0,0) -- (3,0);
\draw [thick][-{Stealth[length=3mm, width=2mm]}](0,3) -- (0,0);
\draw[thick][dashed] (0,3)--(3,0);
\draw[thick][dashed] (5.5,1.5) circle (1.5cm);
\end{tikzpicture}
\]
\caption{The space $\LL$ of linear pencils on the hyperbolic plane $\hh$. After we glue a half-torus to a disk along the dashed line (parabolic pencils) we get the projective space $\R P^2$.}
\label{fig:gluing}
\end{figure}

\end{itemize}

\section{Generalized cycles and extended pencils}
Let $\ca,\cb\approx M, N$. In the hyperbolic case not every combination $xM+yN$ defines a cycle, as the corresponding matrix may be positive or negative definite or even degenerate. Thus, the hyperbolic pencils (and parabolic as well) are not complete lines in the space $\CC$. Hence, in order to fully embrace projective geometry we need to extend the pencils. And for this, in turn, we need to generalize cycles a bit.

A \textit{generalized cycle} is an equation of the form
$$
k z\oline{z} + l z + m\oline{z} + n=0,
$$
where the matrix $M=\left(\begin{smallmatrix}k & l\\ m & n\end{smallmatrix}\right)\in\Herm(2)$ is nonzero and Hermitian, considered up to a nonzero real factor. We use the same notation $\ca\sim M$ for generalized cycles. We keep notation $\ca\approx M$ only for ordinary cycles. However, we naturally extend the orthogonality condition $\ca\perp\cb$ for generalized $\ca,\cb$.

For generalized cycles we have types just as we did for pencils.
When $M$ is non-degenerate and indefinite this is ordinary or \textit{elliptic} cycle. When $M$ has rank $1$ then $M$ uniquely corresponds to the point  $-\frac{m}{k}\in\C$; additionally, the matrix $\left(\begin{smallmatrix}0 & 0\\ 0 & 1\end{smallmatrix}\right)$ corresponds to $\infty$. These are the \textit{parabolic} cycles.  Finally, when $M$ is positive or negative definite it corresponds to a \textit{virtual} or \textit{hyperbolic} cycle. Virtual cycles could be represented by circles of pure imaginary radii; centers are still defined by the same formula. The space of all generalized cycles $\GG$ is naturally homeomorphic to the projective space $\R P^3$. Its universal cover is the space $\widetilde{\GG}$ of oriented generalized cycles, and it is homeomorphic to $S^3$. 
The isometry group $\Iso\M$ acts in the same way on $\GG$ and $\widetilde{\GG}$ and preserves the cycle types. The type is easily recovered from the corresponding 
matrix.

\begin{prop}\label{cycle-type}
Let $\ca\sim M$ be a generalized cycle. Then $\ca$ is elliptic iff $\det M<0$, parabolic iff $\det M=0$, and hyperbolic iff $\det M > 0$.
\end{prop}

All cycles in the sequel will always be assumed to be elliptic, unless otherwise directly stated.

\medskip

An \textit{extended pencil} 
$\GG(\ca,\cb)$ defined by two generalized cycles 
$\ca,\cb\sim M, N$ is a projective line in $\GG$ defined by $\ca,\cb$, i.e.
$$
\GG(\ca,\cb) = \{\cc\sim xM+yN|x,y\in\R\}.
$$
Every ordinary pencil is uniquely completed to an extended one. For example, 
a parabolic pencil is completed by adding a unique parabolic cycle, and a hyperbolic pencil is completed by adding two parabolic cycles at its distinguished points and an open interval of virtual cycles between them. 

Every extended pencil is completed from a uniquely defined ordinary pencil. Indeed, let $\ca,\cb\sim M, N$ be two virtual cycles. We can assume that both $M$ and $N$ are positive definite. Let $x=\trr N$, $y=\trr M$. Then the matrix $xM-yN$ is nonzero since $M$, $N$ are not proportional, and has zero trace. 
Since all Hermitian rank $1$ matrices have nonzero trace it is nondegenerate and indefinite, so it defines an elliptic cycle $\cc$. Moreover, all cycles in some neighborhood of $\cc$ are also elliptic. Therefore, $\GG(\ca,\cb)\supset\gP$ where $\gP$ is an ordinary hyperbolic pencil.

Therefore, extended pencils $\gP$ are in one to one correspondence with the real linear subspaces $W\subset\Herm(2)$ of dimension $2$. The space of pencils $\PP$ has its own topology that transforms it into a compact four-dimensional manifold. We can now identify it.
\begin{theorem}
The space of pencils $\PP$ on the M\"obius plane is homeomorphic to the real Grassmannian $\Gr_2(4)$. 
\end{theorem}
The pencil space $\PP$ is not simply connected and its universal cover $\widetilde{\PP}$ is the space of oriented pencils homeomorphic to the oriented Grassmannian $\Gr^+_2(4)$. The orientation-forgetful map $\widetilde{\PP}\rightarrow\PP$ is a universal two-fold covering. However, we will never need this space, and consider further only non-oriented pencils.

\medskip
The type of a pencil can be easily identified from the corresponding subspace.
\begin{prop}\label{W-type}
Let $W\subset\Herm(2)$ be a subspace of dimension 2. Then the corresponding pencil $\gP$ is elliptic iff $W$ is positive-definite; it is hyperbolic iff $W$ is indefinite, and parabolic iff $W$ is degenerate. 
Moreover, the map $\gP\mapsto \gP^\perp$ is exactly the orthogonal complement map $W\mapsto W^\perp$.
\end{prop}
\begin{proof}
Let $\ca,\cb\approx M, N$ be two distinct cycles. Then $\gP=\CC(\ca,\cb)$ is elliptic iff $|\la\ca,\cb\ra|<1$, parabolic iff $\la\ca,\cb\ra=\pm1$, and hyperbolic iff $|\la\ca,\cb\ra|>1$. On the other hand, the determinant of the Gram matrix
$$
\left|
\begin{array}{cc}
\la M,M\ra & \la M,N\ra\\
\la N,M\ra & \la N,N\ra
\end{array}
\right|=
\left|
\begin{array}{cc}
1 & \la\ca,\cb\ra\\
\la\ca,\cb\ra & 1
\end{array}
\right|= 1 - \la\ca,\cb\ra^2
$$ 
identifies the signature of the corresponding subspace $W$. We have $\sig W=(2,0)$ (positive-definite $W$) when the determinant is positive, $\sig W=(1,0)$ (degenerate $W$) when it is zero, and $\sig W=(1,1)$ (indefinite $W$) otherwise. These cases uniquely correspond to the pencil types.
\end{proof}

\medskip
Consider generalized cycles in the standard geometries.

\begin{itemize}
\item in spherical geometry all linear pencils are elliptic, and nothing needs to be added.

\item in Euclidean geometry linear pencils are either elliptic or parabolic.
Parabolic pencils are completed by adding one point at infinity for each direction. The line at infinity is the only non-elliptic (parabolic) cycle that is added.

\item in hyperbolic geometry all types are present. Parabolic cycles once again are in bijection with the line at infinity $\oline{\R}$. And we also have hyperbolic cycles. A real symmetric matrix $M=\left(\begin{smallmatrix}k & l\\ l & n\end{smallmatrix}\right)$ defines a hyperbolic cycle iff $\det M>0$. It is convenient to assume that this cycle is a point $z=-\frac{l}{k}+i\frac{\sqrt{\det M}}{k}\in\mathcal{H}$. Conversely, a point $z=x+iy$ corresponds to a unique hyperbolic cycle $N=\left(\begin{smallmatrix}-1 & x\\ x & -x^2-y^2\end{smallmatrix}\right)$. Moreover, for $\ca\sim M$ we have
\begin{gather*}
\la M,N\ra = (-k+2lx-n(x^2+y^2)) - (k+n)(-1-x^2-y^2)=\\
k(x^2+y^2) + 2lx + n=kz\oline{z} + lz+l\oline{z} + n.
\end{gather*}
Therefore, $z\perp \ca$ iff $z\in \ca$. Thus, a hyperbolic pencil $\gP$ is completed by adding two points at infinity and all points from the unique line $\cl\in\gP^\perp$.

Hence, as in spherical case, we have a bijection between linear pencils and generalized cycles.
\end{itemize}

\section{The splitting factor}
Every extended pencil $\gP$ is homeomorphic to a projective line $\R P^1$.
We introduce an invariant coordinate on it. Let $\ca,\cb\approx M, N$ be two oriented elliptic cycles. Then for every $\cc\in \GG(\ca,\cb)$ there exists a unique $\lambda\in\oline{\R}$ such that $\cc\sim M-\lambda N$. We take $\lambda=\infty$ when $\cc\cong\cb$. The \textit{splitting factor} $(\ca,\cb;\cc)$ is this number $\lambda$. It is uniquely defined and invariant. Moreover, $(\ca,\cb;\cc)$ is odd in the first two arguments and even in the third. The map
$\cc\mapsto (\ca,\cb;\cc)$
is a homeomorphism between $\GG(\ca,\cb)$ and $\oline{\R}$.

\begin{prop}\label{splitting-prop}
Let $\ca, \cb$ be two distinct cycles. Let $\cc\in\GG(\ca,\cb)$. Then
$$
(\cb,\ca;\cc) = (\ca,\cb;\cc)^{-1}.
$$
\end{prop}
\begin{proof}
Let $\ca,\cb\approx M,N$. Let $\lambda=(\ca,\cb;\cc)$. Then $\cc\sim M-\lambda N$,
and also $\cc\sim N-\lambda^{-1}M$.
\end{proof}

\begin{prop}\label{(c,a;b)}
Let $(\ca,\cb;\cc)=\lambda$ and $\ca,\cb\approx M, N$.
If $\cc$ is elliptic then
$\cc\approx(\cc,\ca;\cb)(M-\lambda N)$.
\end{prop}

\begin{proof}
We have that $\cc\sim M-\lambda N$. Let $\cc\approx L$. Then $xL = M-\lambda N$ for some $x\in\R$. Then $\cb\sim M-xL$ implying $x=(\ca,\cc;\cb)$.
\end{proof}

\begin{prop}
Let $\ca,\cb,\cc$ be three distinct oriented collinear elliptic cycles. Then
$$
(\ca,\cb;\cc)(\cb,\cc;\ca)(\cc,\ca;\cb) = -1.
$$
\end{prop}
\begin{proof}
Let $\ca,\cb\approx M, N$ and $\lambda=(\ca,\cb;\cc)$. Then $\cc\approx(\cc,\ca;\cb)(M-\lambda N)$
and $\cc\approx(\cc,\cb;\ca)(N-\lambda^{-1}M)$. Since the normalized matrix is uniquely defined we obtain
$$
(\cc,\ca;\cb) = -  \lambda^{-1}(\cc,\cb;\ca),
$$
and the rest follows.
\end{proof}

\begin{prop}
Let $\ca,\cb,\cc$ be elliptic and collinear. Then
$$
\la \ca,\cc\ra - \la\cb,\cc\ra  (\ca,\cb;\cc) = (\ca,\cc;\cb).
$$
\end{prop}
\begin{proof}
Let $\ca,\cb\approx M, N$. Then $\cc\approx(\cc,\ca;\cb)(M-\lambda N)$. Therefore
\begin{gather*}
\la\ca,\cc\ra = \la M, (\cc,\ca;\cb)(M-\lambda N)\ra =
(\cc,\ca;\cb)(\la M,M\ra - \lambda\la M,N\ra)=\\
(\cc,\ca;\cb)-(\cc,\ca;\cb)(\ca,\cb;\cc)\la\ca,\cb \ra=
(\cc,\ca;\cb)+(\cc,\cb;\ca)\la\ca,\cb \ra.
\end{gather*}
Then just swap $\ca$ and $\cc$.
\end{proof}

\begin{corollary}\label{lambda-explicit}
Let $\ca,\cb,\cc$ be elliptic and lie on a non-parabolic pencil. Then
$$
(\ca,\cb;\cc) = \frac{\la\ca,\cb\ra-\la\ca,\cc\ra\la\cb,\cc\ra}{1-\la\cb,\cc \ra^2}.
$$
\end{corollary}
\begin{proof}
We have that 
$$
(\ca,\cb;\cc) = \la \ca,\cb\ra - \la\cb,\cc\ra  (\ca,\cc;\cb) = 
 \la \ca,\cb\ra - \la\cb,\cc\ra(\la \ca,\cc\ra - \la\cb,\cc\ra  (\ca,\cb;\cc)).
$$
\end{proof}

\begin{prop}\label{lambda-line}
Let three circles $\ca,\cb,\cc$ be collinear and have distinct centers $O_a, O_b, O_c$ and signed radii $r_a, r_b, r_c$. Then
$$
(\ca,\cb;\cc)= \frac{O_c-O_a}{O_c-O_b}\frac{r_b}{r_a}.
$$
\end{prop}
\begin{proof}
Let $\lambda=(\ca,\cb;\cc)$.
Let $\ca\approx M=\left(\begin{smallmatrix}k_a & l_a\\ m_a & n_a\end{smallmatrix}\right)$
and $\cb\approx N=\left(\begin{smallmatrix}k_b & l_b\\ m_b & n_b\end{smallmatrix}\right)$. Then $O_a=-\frac{m_a}{k_a}$, $O_b=-\frac{m_b}{k_b}$, $r_a=\frac{1}{k_a}$, $r_b=\frac{1}{k_b}$. Furthermore, $\cc\sim M-\lambda N$ and $O_c=-\frac{m_a-\lambda m_b}{k_a-\lambda k_b}$. Therefore,
$$
\frac{O_c-O_a}{O_c-O_b}\frac{r_b}{r_a}=
\frac{\frac{m_a-\lambda m_b}{k_a-\lambda k_b}-\frac{m_a}{k_a}}{\frac{m_a-\lambda m_b}{k_a-\lambda k_b}-\frac{m_b}{k_b}}\frac{k_a}{k_b}=
\frac{k_a(m_a-\lambda m_b)-(k_a-\lambda k_b)m_a}{k_b(m_a-\lambda m_b)-(k_a-\lambda k_b)m_b}=
\frac{-\lambda k_a m_b+\lambda k_b m_a}{k_b m_a-k_am_b}=\lambda.
$$

\end{proof}


\section{Digons}
A \textit{digon} is a part of the plane bounded by two
cycle segments sharing common endpoints. A digon, therefore, has two vertices and two sides (Figure~\ref{fig:digons}).

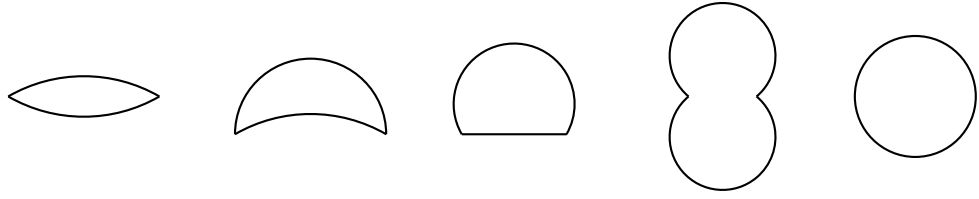
\begin{figure}[h]
\[
\begin{tikzpicture}
\draw [thick](0,0.5) arc (120:60:2);
\draw [thick](0,0.5) arc (-120:-60:2);
\draw [thick](3,0) arc (180:0:1);
\draw [thick](3,0) arc (120:60:2);
\draw [thick](6,0) arc (210:-30:0.8);
\draw [thick](6,0) -- (7.39, 0);
\draw [thick](9,0.5) arc (230:-50:0.7);
\draw [thick](9,0.5) arc (-230:50:0.7);
\draw [thick](12,0.5) circle (0.8cm);
\end{tikzpicture}
\]
\caption{Various kinds of M\"obius digons.}
\label{fig:digons}
\end{figure}

Moreover, the angles between the sides at the vertices are equal. 
Therefore, for every digon its \textit{inner angle} $\alpha$ is uniquely defined.
The inner angle $\alpha$ can be any number from the interval $(0,2\pi)$.
The complement of a digon with an inner angle $\alpha$ is a digon of angle $2\pi-\alpha$.
 When $\alpha=\pi$ then the corresponding digon is a half-plane. This is the only case when the sides of a digon belong to a single cycle and vertices cannot be identified unless explicitly given. Such digon is \textit{degenerate}.
 We call a digon \textit{proper} when $\alpha<\pi$. The sides of a proper digon are always non-collinear.

\begin{theorem}
Two digons are congruent iff their inner angles are the equal.
\end{theorem}
\begin{proof}
Without loss of generality we can assume that both digons have $\infty$ as one of the vertices.
Then both of them are ordinary Euclidean angles. Then there exists a Euclidean similarity taking one to another iff their numeric values are the same.
\end{proof}

Any two intersecting cycles $\ca,\cb$ split the plane into four proper digons having angles $\alpha,\pi-\alpha, \alpha, \pi-\alpha$ for some $\alpha$. When the cycles are oriented exactly one of these digons appears on the left half-plane for both $\ca$ and $\cb$, and we will always consider it as the one bounded by $\ca$ and $\cb$. Conversely, when considering a digon we will always assume its boundary is oriented and it is located on the left for both boundary cycles. Note that we have then $\angle(\ca,\cb)=\pm(\pi-\alpha)$ and $\la\ca,\cb\ra=-\cos\alpha$ where $\alpha$ is the inner angle.

A \textit{bisector} of a digon is a cycle $\cl$ passing through its vertices that splits the digon into two equal halves and such that the inversion $\sigma_{\cl}$ about it swaps them.
\begin{prop}
Every digon has a unique bisector.
\end{prop}
\begin{proof}
Again, we can assume that one vertex of a digon is $\infty$.
The bisector, if it exists must pass through both vertices, so it is also a straight line. Clearly, only the Euclidean bisector satisfies the required conditions.
\end{proof}
Therefore the bisector splits a digon of angle $\alpha$ into two equal digons of angle $\alpha/2$. The bisector is, of course, one of the midcycles. The other midcycle is the \textit{external bisector}.

More generally, a \textit{cevian} of a digon is any cycle $\cc$ passing through its vertices, i.e. an element of the pencil formed by its sides.  
Every cevian $\cc$ is uniquely defined by the angles it forms with the sides of digon. Moreover, for every cevian $\cc$ there exists a unique cevian $\cc^\perp$ orthogonal to it. 
 
\begin{prop} 
Let $\cc\in\CC(\ca,\cb)$ be a cevian. Then
$$
(\ca,\cb;\cc) = \frac{\sin\angle(\ca,\cc)}{\sin\angle(\cb,\cc)}
$$
where the oriented angles are both measured at the same vertex of the digon.
\end{prop}
\begin{proof}
Let $\alpha=\angle(\ca,\cb)$, $\alpha_1=\angle(\ca,\cc)$, $\alpha_2=\angle(\cb,\cc)$ all be measured at one of the vertices of the digon. We have $\alpha = \alpha_1-\alpha_2$ modulo $2\pi$, and according to Corollary~\ref{lambda-explicit} we have
$$
(\ca,\cb;\cc) = \frac{\la\ca,\cb\ra-\la\ca,\cc\ra\la\cb,\cc\ra}{1-\la\cb,\cc \ra^2}=
\frac{\cos\alpha-\cos\alpha_1\cos\alpha_2}{1-\cos^2\alpha_2} = 
\frac{\sin\alpha_1\sin\alpha_2}{\sin^2\alpha_2}
=\frac{\sin\alpha_1}{\sin\alpha_2}.
$$
Note that $\alpha$ is not the inner angle in this argument.
\end{proof}
 
Moreover, we have $(\ca,\cb;\cc)>0$ iff $\cc$ splits the digon in two, and $(\ca,\cb;\cc)<0$ iff $\cc$ intersects the digon only in its vertices. Moreover, two cevians $\cc,\cd$ are symmetric with respect to $\cl$ iff $(\ca,\cb;\cd)=(\ca,\cb;\cc)^{-1}$.

\begin{prop}
The map $\cc\mapsto(\ca,\cb;\cc)$ is a bijection between the (non-oriented) cevians and $\oline{\R}$. Let $\lambda = (\ca,\cb;\cc)$. Then $\cc\cong\ca$ iff $\lambda=0$, $\cc\cong\cb$ iff $\lambda=\infty$, $\cc\cong\cl$ is the bisector iff $\lambda=1$, and
$\cc\cong\cl^\perp$ is the external bisector iff $\lambda=-1$.
\end{prop}
\begin{proof}
Let $A$ be one of the vertices. Let $\alpha$ be the inner angle.
Select orientation on every cevian $\cc$
such that the angle $\theta= \angle(\cb, \cc)$ measured in $A$ is nonnegative. Then $\angle(\ca,\cc)=\pi-\alpha+\theta$ and every cevian uniquely corresponds to some $\theta\in[0,\pi)$. We have then
$$
\lambda = \frac{\sin(\alpha-\theta)}{\sin\theta} = 
\sin\alpha\cot\theta -\cos\alpha.
$$
But this is a bijection between $[0,\pi)$ and $\oline{\R}$. Moreover, $\cb,\cl,\ca,\cl^\perp$ correspond to $\theta=0$, $\alpha/2$, $\alpha$, $(\pi+\alpha)/2$ that, in turn, correspond to $\lambda=\infty$, $1$, $0$, $-1$ respectively.
\end{proof}

\begin{prop}\label{sin-alpha}
Assume that $\frac{\sin\alpha_1}{\sin\alpha_2}=\lambda$, and $\alpha_1+\alpha_2=\alpha$. Then
$$
\sin\alpha_1 = \varepsilon\frac{\lambda\sin\alpha}{\sqrt{1+2\lambda\cos\alpha+\lambda^2}},
\quad
\sin\alpha_2 = \varepsilon\frac{\sin\alpha}{\sqrt{1+2\lambda\cos\alpha+\lambda^2}}.
$$
Where $\varepsilon=\pm 1$.
\end{prop}
\begin{proof}
We have that
$$
\cos\alpha_1\cos\alpha_2 = \cos\alpha + \sin\alpha_1\sin\alpha_2.
$$
Let $x=\sin\alpha_2$. Then after squaring both sides we obtain
$$
(1-\lambda^2x^2)(1-x^2) = \cos^2\alpha + 2\lambda\cos\alpha x^2 + \lambda^2 x^4.
$$
This is equivalent to
$$
\sin^2\alpha = (1+2\lambda\cos\alpha + \lambda^2) x^2.
$$
\end{proof}

\begin{corollary}\label{acb}
Let $\ca$ and $\cb$ intersect, and $\lambda=(\ca,\cb;\cc)$. Then
$(\ca,\cc;\cb) = \pm \sqrt{1+2\lambda\cos\alpha + \lambda^2}$, where $\alpha$ is the inner angle of the digon bounded by $\ca$ and $\cb$.
\end{corollary}
\begin{proof}
Let $\alpha_1=\angle(\ca,\cc)$, $\alpha_2=\angle(\cc,\cb)$. Then
$\lambda=-\frac{\sin\alpha_1}{\sin\alpha_2}$, and
$(\ca,\cc;\cb)=\frac{\sin\angle(\ca,\cb)}{\sin\angle(\cc,\cb)}=
\frac{\sin(\alpha_1+\alpha_2)}{\sin\alpha_2}=
\pm\sqrt{1-2\lambda\cos(\alpha_1+\alpha_2)+\lambda^2}$ according to Proposition~\ref{sin-alpha}.
If $\alpha$ is the inner angle then $\cos\alpha=-\cos\angle(\ca,\cb)=
-\cos(\alpha_1+\alpha_2)$. This completes the proof. 
\end{proof}

\section{Annuli}
Consider two non-intersecting cycles $\ca$ and $\cb$.
Let $\Ann(\ca,\cb)\subset\M$ be the annulus bounded by them.

\begin{prop}[\cite{J}]
There exists an inversion taking $\ca$ and $\cb$ to a pair of concentric circles. 
\end{prop}
\begin{corollary}
Every annulus is congruent to a ring -- a region bounded by two concentric circles.
\end{corollary}

For a ring bounded by radii $r<R$ its \textit{modulus} is the positive number $\mu=\frac{1}{2\pi}\ln\frac{R}{r}$. It is well-known, that modulus is uniquely defined for every annulus~\cite{Ahl,Neh}. Moreover, Proposition~\ref{cs} implies that 
$\la\ca,\cb\ra=\pm\cosh 2\pi\mu$, where $\mu=\mu(\Ann(\ca,\cb))$. 
The midcycle of $\ca$ and $\cb$ is the \textit{bisector} of the annulus. It splits it into two equal annuli of modulus $\frac{1}{2}\mu$.

When $\ca$ and $\cb$ are disjoint and oriented the \textit{oriented modulus} is defined by
$$
\mu(\ca,\cb) = \pm \mu(\Ann(\ca,\cb))
$$
where the sign $+$ is taken when the annulus lies in the left half-plane of $\ca$, and $-$ is taken otherwise. This is an invariant under $\Iso^+\M$ function, odd in the first argument, and even in the second. Orientation-reversing isometries change the sign of $\mu(\ca,\cb)$ just like for $\sin\angle(\ca,\cb)$.

We have some trivial properties.
\begin{prop}
Let $\ca$ and $\cb$ be disjoint cycles. Then $\mu(\ca,\cb)=\mu(\cb,\ca)$ when the annulus $\Ann(\ca,\cb)$ appears in the same half-plane (left or right) for both cycles, and
$\mu(\ca,\cb)=-\mu(\cb,\ca)$ otherwise. 
\end{prop}

\begin{prop}
Let $\ca,\cb$ be disjoint and $\cc\in\CC(\ca,\cb)$. If the annulus $\Ann(\ca,\cb)$ appears on the same side (left of right) for both $\ca,\cb$ then 
$$
\mu(\ca,\cb) = \mu(\ca,\cc)+\mu(\cb,\cc).
$$
Otherwise,
$$
\mu(\ca,\cb) = \mu(\ca,\cc)-\mu(\cb,\cc).
$$
\end{prop}

We can now summarize the properties of the M\"obius cosine.
\begin{prop}
Let $\ca,\cb$ be two oriented cycles. Let $\xi=\la\ca,\cb\ra$. Then
\begin{enumerate}
\item[$(i)$] if $|\xi|<1$ then $\ca$ and $\cb$ intersect and $\xi=\cos\angle(\ca,\cb)$. In particular, $\xi=0$ iff $\ca\perp\cb$.
\item[$(ii)$] if $|\xi|=1$ then $\ca$ and $\cb$ are tangent.
More specifically, $\xi=-1$ iff the monogon bounded by $\ca$ and $\cb$ appears on the same half-plane (left or right) for both $\ca$ and $\cb$, and $\xi=1$ otherwise.

\item[$(iii)$] if $|\xi|>1$ then $\ca$ and $\cb$  are disjoint and $\xi=\varepsilon\cosh2\pi\mu(\ca,\cb)$, where $\varepsilon=-1$ when the annulus $\Ann(\ca,\cb)$ appears in the same half-plane (left or right) for both of them, and $\varepsilon=1$ otherwise. 
\end{enumerate}
\end{prop}
\begin{proof} Assume that $\ca$ and $\cb$ are circles. The case $(i)$ follows from the definition of the angle between circles.

Let $r_a$, $r_b$ be the signed radii of $\ca$ and $\cb$ respectively. Let $d\ge0$ be the distance between the centers.
Assume $\xi=1$. Then $r_a^2 + r_b^2 - d^2 = 2r_ar_b$, or equivalently, $d = |r_a-r_b|$. If $\ca$ and $\cb$ are internally tangent, then $r_a$ and $r_b$ have the same sign, and the monogon appears on different sides for both of them.
When the tangency is external, then $r_a$ and $r_b$ have opposite signs, and the monogon appears on the different sides for both cycles again. The case $\xi=-1$ is considered analogously. This settles $(ii)$. 

For $(iii)$ assume that $\ca$ and $\cb$ are concentric. Then $\xi = \frac{1}{2}(\frac{r_a}{r_b}+\frac{r_b}{r_a})$.
The modulus is $\mu=\mu(\Ann(\ca,\cb))=\frac{1}{2\pi}\big|\ln|\frac{r_a}{r_b}|\big|$.
When both circles have positive or negative orientation
then $r_a$ and $r_b$ have same sign and
$$
\la\ca,\cb\ra= \cosh2\pi\mu = \cosh2\pi \mu(\ca,\cb).
$$ 
The remaining case is considered analogously.
\end{proof}
A \textit{cevian} of the annulus $\Ann(\ca,\cb)$ is any element of the pencil $\CC(\ca,\cb)$.

\begin{prop} 
Let $\ca,\cb$ be disjoint, let $\cc\in\CC(\ca,\cb)$ be a cevian. Then
$$
(\ca,\cb;\cc) = \frac{\sinh 2\pi\mu(\ca,\cc)}{\sinh 2\pi\mu(\cb,\cc)}
$$
\end{prop}
\begin{proof}
Assume that $\Ann(\ca,\cb)$ appears on the same side for both $\ca$, $\cb$. Then $\la\ca,\cb\ra=-\cosh2\pi\mu(\ca,\cb)$, and 
$\la\ca,\cc\ra\la\cb,\cc\ra=-\cosh2\pi\mu(\ca,\cc)\cosh2\pi\mu(\cb,\cc)$ for any orientation of $\cc$. Moreover, we have $\mu(\ca,\cb)=\mu(\ca,\cc)+\mu(\cb,\cc)$.
Then according to Proposition~\ref{lambda-explicit} we have
\begin{gather*}
(\ca,\cb;\cc)=\frac{\la\ca,\cb \ra-\la\ca,\cc\ra\la\cb,\cc\ra}{1-\la\cb,\cc \ra^2}=
\frac{\cosh2\pi\mu(\ca,\cb) - 
\cosh2\pi\mu(\ca,\cc)\cosh2\pi\mu(\cb,\cc)}{\cosh^22\pi\mu(\cb,\cc)-1}=\\
=\frac{\sinh2\pi\mu(\ca,\cc)\sinh2\pi\mu(\cb,\cc)}{\sinh^22\pi\mu(\cb,\cc)}
=\frac{\sinh2\pi\mu(\ca,\cc)}{\sinh2\pi\mu(\cb,\cc)}.
\end{gather*}
The remaining case is considered analogously.
\end{proof}

\begin{prop}
Let $\ca,\cb$ be disjoint.
The map $\cc\mapsto(\ca,\cb;\cc)$ is a bijection between the (non-oriented) cevians and $\oline{\R}\setminus[\la\ca,\cb\ra-m,\la\ca,\cb\ra+m]$, where $m=|\sinh2\pi\mu(\ca,\cb)|$. Let $\lambda = (\ca,\cb;\cc)$. Then $\cc\cong\ca$ iff $\lambda=0$, $\cc\cong\cb$ iff $\lambda=\infty$, and $\cc\cong\cl$ is the bisector iff $\lambda=1$ for $\la\ca,\cb\ra<0$ or $\lambda=-1$ for $\la\ca,\cb\ra>0$. When $\lambda$ is approaching the endpoints of the forbidden interval then $\cc$ shrinks to one of its distinguished points.
\end{prop}
\begin{proof}
Let $\ca,\cb\approx M, N$. Then $\cc\sim M-\lambda N$, and
$$
\det (M-\lambda N) = -\la M-\lambda N, M-\lambda N \ra=
-1 +2\lambda\la\ca,\cb\ra-\lambda^2.
$$
Therefore, $\cc$ is hyperbolic iff $\lambda^2 -2\lambda\la\ca,\cb\ra+1<0$, and this is equivalent to
$$
|\lambda-\la\ca,\cb\ra| < \sqrt{\la\ca,\cb\ra^2-1} = 
\sqrt{\cosh^22\pi\mu(\ca,\cb)-1}=|\sinh2\pi\mu(\ca,\cb)|. 
$$
Trivially, $\ca$ and $\cb$ correspond to $\lambda=0$ and $\lambda=\infty$. We also have $\mu(\ca,\cl)=\frac{1}{2}\mu(\ca,\cb)$ and $\mu(\cb,\cl)=\frac{1}{2}\mu(\cb,\ca)$.
Therefore, $\cl$ corresponds to $\lambda=1$ when $\mu(\ca,\cb)=\mu(\cb,\ca)$, and this happens when the annulus occurs on the same side for $\ca$ and $\cb$, which is when $\la\ca,\cb\ra<0$.
\end{proof}

The number $|\mu(\ca,\cb)|$ can be treated as a distance between the disjoint cycles $\ca,\cb$. The following \textit{inverse triangle inequality} holds.

\begin{prop}
Let $\ca,\cb,\cc$ be pairwise disjoint and assume that $\cc$ splits $\ca$ and $\cb$. Then
$$
|\mu(\ca,\cb)|\geq |\mu(\ca,\cc)| + |\mu(\cb,\cc)|
$$
and equality holds iff $\cc\in\CC(\ca,\cb)$.
\end{prop}
\begin{proof}
Assume that $\ca,\cb$ are centered at $0$ and have positive radii $r<R$ respectively. Assume that $\cc$ has center $d>0$ and radius $\rho>0$. Since $\cc$ splits $\ca$, $\cb$ we have $d+\rho<R$ and $\rho-d>r$. Then
$$
|\mu(\ca,\cb)| = \frac{1}{2\pi}\ln\frac{R}{r},
$$

$$
|\mu(\ca,\cc)| =  \frac{1}{2\pi}\arcosh\frac{r^2+\rho^2-d^2}{2r\rho}
\leq  \frac{1}{2\pi}\arcosh\frac{r^2+\rho^2}{2r\rho} = 
 \frac{1}{2\pi}\ln\frac{\rho}{r}, 
$$
and
$$
|\mu(\cb,\cc)| =  \frac{1}{2\pi}\arcosh\frac{R^2+\rho^2-d^2}{2R\rho}
\leq  \frac{1}{2\pi}\arcosh\frac{R^2+\rho^2}{2R\rho} 
=  \frac{1}{2\pi}\ln\frac{R}{\rho}. 
$$
The rest follows. An equality holds iff $d=0$ which is equivalent to $\cc\in\CC(\ca,\cb)$.
\end{proof}

In the hyperbolic geometry $|\mu(\ca,\cb)|$ is equal to the distance  $d(\ca,\cb)$ between the disjoint lines $\ca$ and $\cb$. The inverse triangle inequality makes perfect sense: $d(\ca,\cb)$ is the length of the common perpendicular for $\ca$ and $\cb$; it is unique, but may be not orthogonal to $\cc$ too.

We also summarize the analogous results for monogons.
Let $\ca,\cb$ be tangent. A \textit{cevian} of the monogon bounded by $\ca$ and $\cb$ is an element of $\CC(\ca,\cb)$. In this case, unlike elliptic or hyperbolic, we cannot represent $(\ca,\cb;\cc)$ as a ratio of two independent invariant values. This is the reason why the splitting factor was introduced as a separate concept in first place. We can only use the following construction. Take the distinguished point to $\infty$ and place $\ca,\cb,\cc$ to be vertical lines.

\begin{prop}
Let $\ca,\cb,\cc$ be vertical lines intersecting $\R$ in the points $a, b, c$ respectively. Let $\xi = \la\ca,\cb\ra\in\{\pm 1\}$. Then
$$
(\ca,\cb;\cc) = \xi\frac{c-a}{c-b}.
$$
\end{prop}
\begin{proof}
We have that $\ca$ is defined by the equation
$z + \oline{z} - 2a=0$. Assume $\ca\approx
\left(
\begin{smallmatrix}
 0 & 1\\
1 & -2a
\end{smallmatrix}\right)$, 
$\cb\approx
\left(
\begin{smallmatrix}
 0 & 1\\
1 & -2b
\end{smallmatrix}\right)$, 
$\cc\sim
\left(
\begin{smallmatrix}
 0 & 1\\
1 & -2c
\end{smallmatrix}\right)$.
In this case $\xi=1$. We have that
$
\left(
\begin{smallmatrix}
 0 & 1\\
1 & -2a
\end{smallmatrix}
\right)-
\lambda
\left(
\begin{smallmatrix}
 0 & 1\\
1 & -2b
\end{smallmatrix}
\right)
=
\left(
\begin{smallmatrix} 0 & 1-\lambda\\
1-\lambda & -2(a-\lambda b)
\end{smallmatrix}
\right)
$.
Then $(a-\lambda b) = (1-\lambda)c$ is equivalent to $\lambda =\frac{c-a}{c-b}$.

The remaining cases are considered analogously.
\end{proof}

\begin{prop}
Let $\ca,\cb$ be tangent cycles. Let $\xi=\la\ca,\cb\ra\in\{\pm1\}$. The map $\cc\mapsto(\ca,\cb;\cc)$ is a bijection between the (non-oriented) cevians and $\oline{\R}\setminus\{\xi\}$. Let $\lambda = (\ca,\cb;\cc)$. Then $\cc\cong\ca$ iff $\lambda=0$, $\cc\cong\cb$ iff $\lambda=\infty$, and $\cc\cong\cl$ is the bisector of the monogon iff $\lambda=-\xi$. When $\lambda\rightarrow\xi$ the cycle $\cc$ shrinks to the distinguished point.
\end{prop}
\begin{proof}
Once again, assume $\ca,\cb,\cc$ are vertical lines intersecting $\R$ in the points $a, b, c$. The map $c\mapsto\xi\frac{c-a}{c-b}$ is a bijection between $\CC(\ca,\cb)$ and $\oline{\R}\setminus\{\xi\}$. Clearly, $\ca$ and $\cb$ correspond to $c=a$ and $c=b$ that are mapped to $\lambda=0$ and $\lambda=\infty$ respectively. The bisector $\cl$ has coordinate $\frac{1}{2}(a+b)$ and is mapped to
$$
\xi\frac{\frac{a+b}{2}-a}{\frac{a+b}{2}-b}=\xi\frac{b-a}{a-b}=-\xi.
$$
\end{proof}

\section{Triangles}
A \textit{triangle} $\Delta=\triangle ABC$  is a part of the plane bounded by three cycle segments that pairwise join three distinct points $A, B, C$ and share no other common points. We write $BC, AC, AB$ for these segments and $\ca, \cb, \cc$ for the cycles containing them. The points $A$, $B$, $C$ are the \textit{vertices} of the triangle.

Unlike Euclidean geometry the triangles on the M\"obius plane appear in extremely wide variety of shapes and forms (Figure~\ref{fig:triags}). For example, Reuleaux triangle, Ying-Yang curve, or an arbelos are all M\"obius triangles.
 
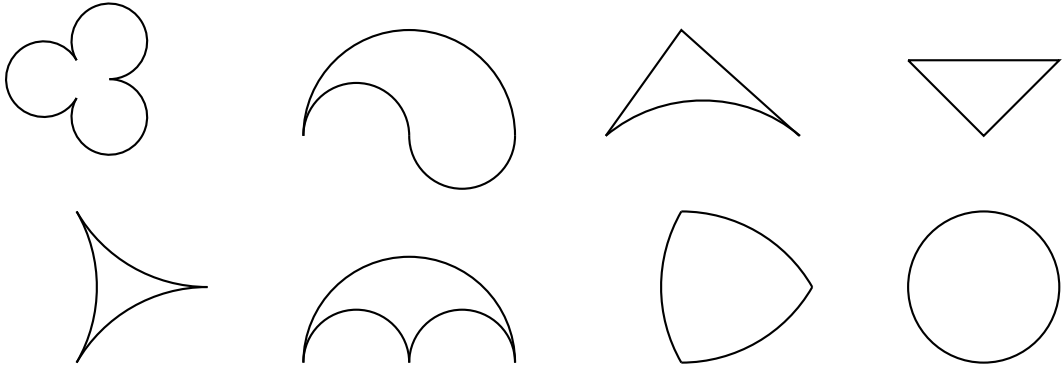
\begin{figure}[h]
\[
\begin{tikzpicture}
\draw [thick](0,0) arc (-30:30:2);
\draw [thick](0,0) arc (150:90:2);
\draw [thick](0,2) arc (-150:-90:2);
\draw [thick](0,4) arc (30:330:0.5);
\draw [thick](0,4) arc (210:-90:0.5);
\draw [thick](0,3.5) arc (-210:90:0.5);
\draw [thick](3,0) arc (180:0:0.7);
\draw [thick](4.4,0) arc (180:0:0.7);
\draw [thick](3,0) arc (180:0:1.4);
\draw [thick](3,3) arc (180:0:0.7);
\draw [thick](4.4,3) arc (180:360:0.7);
\draw [thick](3,3) arc (180:0:1.4);
\draw [thick](8,0) arc (210:150:2);
\draw [thick](8,0) arc (-90:-30:2);
\draw [thick](8,2) arc (90:30:2);
\draw [thick](7,3) arc (130:50:2);
\draw [thick](7,3) -- (8,4.4) -- (9.56,3);
\draw [thick](12,1) circle (1);
\draw [thick](11,4) -- (12,3) -- (13,4) -- (11,4);
\end{tikzpicture}
\]

\caption{Various kinds ot M\"obius triangles.}
\label{fig:triags}
\end{figure} 

Since isometries are conformal the inner angles of a triangle are uniquely defined and are equal for congruent triangles. There are no limitations for a single angle value in this case. 
If $\alpha,\beta,\gamma$ are angles of a M\"obius triangle, then $\alpha, \beta, \gamma\in[0, 2\pi]$. However, not all triples are feasible. The angles must satisfy the \textit{triangle inequalities} that hold for all triangles.
\begin{theorem}[\cite{Epp}]
The angles $\alpha$, $\beta$, $\gamma$ of any M\"obius triangle satisfy the  inequalities
\begin{align*}
-\pi <\beta+\gamma-\alpha<3\pi,\\
-\pi <\alpha+\gamma-\beta<3\pi,\\
-\pi <\alpha+\beta-\gamma<3\pi.
\end{align*}
Any triple of angles satisfying the inequalities  above
can be realized as internal angles of some M\"obius triangle. 
\end{theorem}
\begin{corollary}
For any $\alpha,\beta,\gamma\in[0,\pi)$ or any $\alpha,\beta,\gamma\in(\pi,2\pi]$ 
there exists a M\"obius triangle having $\alpha$, $\beta$, $\gamma$ as its internal angles. 
\end{corollary}

Thus, the feasible triples of angles form a polyhedron filling exactly three quarters of the cube $[0,2\pi]^3$. For example, there is no triangle with angles $0$ and $2\pi$ at the same time, and no triangle with angles $0$, $0$, $\pi$. The polyhedron is central symmetric, as the complement of every triangle is also a triangle sharing same vertices and having angles $2\pi-\alpha$, $2\pi-\beta$, $2\pi-\gamma$.

Note that $\alpha=\pi$ does not mean that $AB$ and $AC$ lie on the same cycle; it only means that $\cb$ and $\cc$ are tangent in $A$ (Figure~\ref{fig:triags}). A triangle degenerates into a digon when two angles are the same and the third is $\pi$. When all angles are $\pi$ the triangle becomes a half-plane. Such triangles are \textit{degenerate}. For non-degenerate triangles $\ca,\cb,\cc$ are non-collinear. 

When considering triangles we will always assume its boundary is oriented and the triangle appears on the left for every bounding segment. If we walk the boundary in this orientation then every vertex uniquely defines the next one; this creates a map which is a permutation of the vertices. This permutation (it is one of the two cycles of order $3$) is the \textit{orientation} of the triangle.

\begin{theorem}\label{triangle-location}
For any three distinct points $A, B, C\in\M$, any angles $\alpha,\beta,\gamma\in[0,2\pi]$ satisfying the triangle inequalities, and any prescribed orientation, there exists a unique triangle with vertices $A, B, C$, internal angles $\alpha, \beta,\gamma$ respectively, and the prescribed orientation.
\end{theorem}
\begin{proof}
Let $\cs$ be the \textit{circumcircle} of $\triag ABC$, i.e. the unique cycle passing through $A$, $B$, $C$. For the given triangle orientation select the opposite orientation on $\cs$. We show
the triangle $\triag ABC$ is uniquely obtained from the right half-plane of $\cs$ by attaching or cutting away three digons (Figure~\ref{fig:triangle}).

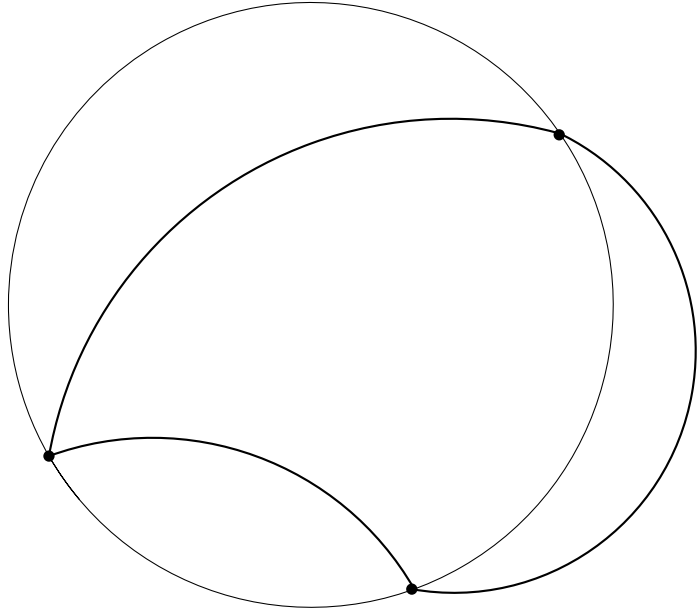
\begin{figure}[h]
\[
\begin{tikzpicture}
\tikzstyle{vertex}=[circle, fill, inner sep=1.5pt]
\node [vertex] at (4,5) {};
\node [vertex] at (8.8,3.24) {};
\node [vertex] at (10.75,9.25) {};

\draw (4,5) arc (210:580:4);
\draw [thick] (4,5) arc (110:30:4);
\draw [thick] (8.8,3.24) arc (-100:65:3.2);
\draw [thick] (4,5) arc (170:75:5.40);

\end{tikzpicture}
\]
\caption{Every M\"obius triangle is sculpted from a half-plane by cutting away or gluing on three proper digons.}
\label{fig:triangle}
\end{figure}

Indeed, $\cs$ must form three digons with $\ca$, $\cb$, $\cc$. Let $\alpha_0$, $\beta_0$, $\gamma_0$ be the signed inner angles of these digons; the sign is chosen to be $+$ for the digon lying inside the left half-plane of $\cs$, and $-$ for the right half-plane. Hence, the values $\alpha_0$, $\beta_0$, $\gamma_0$ uniquely define the location of the segments $AB$, $BC$, $AC$, and, therefore, uniquely define the whole triangle.

The inner angles satisfy the system of linear equations
$$
\begin{cases}
\alpha+\beta_0+\gamma_0=\pi,\\
\alpha_0+\beta+\gamma_0=\pi,\\
\alpha_0+\beta_0+\gamma=\pi,
\end{cases}
$$    
that has a unique solution
\begin{gather*}
\alpha_0 = \frac{\pi+\alpha-\beta-\gamma}{2},\\
\beta_0 = \frac{\pi-\alpha+\beta-\gamma}{2},\\
\gamma_0 = \frac{\pi-\alpha-\beta+\gamma}{2}.
\end{gather*}
The triangle inequalities are exactly the conditions for the digons to be proper
$$
-\pi< \alpha_0,\beta_0, \gamma_0< \pi
$$
as otherwise we will encounter self-intersection for the bounding segments \cite{Epp}. When the conditions are met the triangle is uniquely reconstructed.

To get the unique triangle of the opposite direction we apply the similar procedure 
to the left half-plane of $\cs$, or simply execute the inversion $\sigma_\cs$ about $\cs$ to the first triangle. 
\end{proof}

Therefore, all the theorems and properties of the triangles can be formulated 
in the terms of its angles. Furthermore, there is only one triangle congruence condition for the M\"obius triangles.
\begin{corollary}
Two triangles $\Delta$ and $\Delta^\pr$ are congruent iff their corresponding angles are equal.
\end{corollary}
\begin{proof}
We need to prove sufficiency only. Let $\Delta$ and $\Delta^\pr$ have equal angles.
Let $\sigma\in\Iso M$ be a M\"obius transformation that takes vertices of $\Delta^\pr$ to the corresponding vertices of $\Delta$. Then $\sigma(\Delta^\pr)$ and $\Delta$ share the same vertices and angles. If their orientations are different then we additionally execute inversion about the circumcircle of $\Delta$.
In the end we obtain two triangles sharing the same vertices, angles, and orientation. According to Theorem~\ref{triangle-location}
they coincide.
\end{proof}

We call a triangle $\Delta$ \textit{proper} if its angles $\alpha,\beta,\gamma$ lie in the interval $[0,\pi)$. Proper triangles are always non-degenerate.
When $\alpha+\beta+\gamma<\pi$ then $\Delta$ can be realized as a hyperbolic triangle on $\hh$. If $\alpha,\beta,\gamma$ are positive and $\alpha+\beta+\gamma=\pi$ then there exists a congruent Euclidean triangle. However, if $\alpha+\beta+\gamma>\pi$ this alone does not imply that the triangle is spherical, as there are additional conditions.

\begin{prop}
Let $\triag ABC$ be a proper triangle on the sphere. Then its angles $\alpha$, $\beta$, $\gamma$ satisfy the triangle inequalities
\begin{align*}
\alpha &> \beta + \gamma-\pi,\\
\beta &> \alpha + \gamma-\pi,\\
\gamma &> \alpha + \beta-\pi.
\end{align*}
\end{prop}
\begin{proof}
Assume that the sphere has unit radius.
Let $a, b, c$ be the triangle's sides. Then the standard triangle inequalities hold: $a<b+c$, $b<a+c$, $c<a+b$. Consider now the dual triangle, i.e. a triangle whose vertices are the poles of the sides of the original triangle, and whose sides are great circle arcs connecting these poles. Its angles are $\pi-a$, $\pi-b$, $\pi-c$ and sides are $\pi-\alpha$, $\pi-\beta$, $\pi-\gamma$. Apply the standard triangle inequalities to the dual triangle.
\end{proof}
What do these conditions mean on the M\"obius plane? As we show below a triangle is spherical when $\alpha+\beta+\gamma>\pi$ and it appears completely on the same side for every bounding cycle $\ca,\cb,\cc$, so that the extension of every side does not intersect it.
\begin{theorem}\label{a=b+c-Pi}
Let $\Delta=\triag ABC$ be a proper triangle. The cycle $\ca$ does not split the triangle, i.e. $\ca\cap\Delta = BC$ iff $\alpha>\beta+\gamma-\pi$.
The cycle $\ca$ passes through the vertex $A$ (i.e. $\ca$ is the circumcircle of $\Delta$) iff $\alpha=\beta+\gamma-\pi$. Otherwise $\ca$ intersects and cuts the triangle into two regions.
\end{theorem}
\begin{proof}
Assume that all angles are positive. Let $\cb$ and $\cc$ be straight lines.
Obviously, Euclidean triangles  satisfy the theorem. Therefore, assume further that $\ca$ is a circle. Then the cases are distinguished by the relative position of $A$: if $A$ lies inside $\ca$ then it does not split the triangle (Figure~\ref{fig:proper}(a)), if $A$ lies on $\ca$ then $\ca$ is the circumcircle (Figure~\ref{fig:proper}(b)), and if $A$ lies outside $\ca$ then it splits the triangle (Figure~\ref{fig:proper}(c)).  
Let $B^\pr,C^\pr$ be the other intersection points of $\ca$ with $\cb$ and $\cc$.
Then $\beta=\frac{1}{2}\widearc{BB^\prime}$, $\gamma=\frac{1}{2}\widearc{C^\pr C}$ (we measure the arclength $\widearc{XY}$ while moving from $X$ to $Y$ in positive direction).

In case $(a)$ we have 
$\alpha=\frac{1}{2}(\widearc{BC}+\widearc{B^\pr C^\pr})>
\frac{1}{2}(\widearc{BC}-\widearc{B^\pr C^\pr}) 
= \frac{1}{2}(\widearc{BB^\pr}+\widearc{C^\pr C}-2\pi)=\beta+\gamma-\pi$.

In case $(b)$ we have
$\alpha = \frac{1}{2}\widearc{BC} =\frac{1}{2}(\widearc{BA}+\widearc{A C}-2\pi)=\beta+\gamma-\pi$.

In case $(c)$ we have
$\alpha=\frac{1}{2}(\widearc{BC}-\widearc{C^\pr B^\pr})<
\frac{1}{2}(\widearc{BC}+\widearc{C^\pr B^\pr}) 
= \frac{1}{2}(\widearc{BB^\pr}+\widearc{C^\pr C}-2\pi)=\beta+\gamma-\pi$.

Finally, a similar consideration works when some of the angles are zero.
\end{proof}

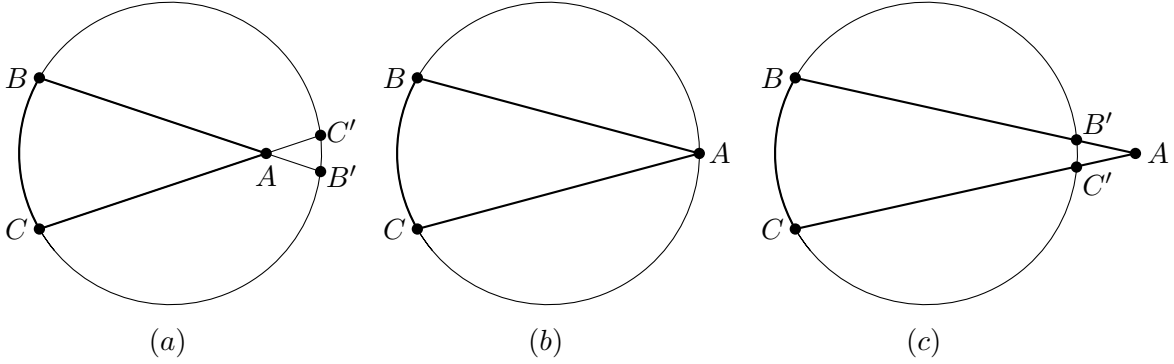
\begin{figure}[h]
\[
\begin{tikzpicture}
\tikzstyle{vertex}=[circle, fill, inner sep=1.5pt]
\draw [thick](2,5) arc (210:150:2);
\draw (2,5) arc (210:580:2);

\node [vertex] at (2,5) {};
\node [vertex] at (2,7) {};
\node [vertex] at (5,6) {};
\node [vertex] at (5.72,6.24) {};
\node [vertex] at (5.72,5.76) {};

\node at (1.7,5) {$C$};
\node at (1.7,7) {$B$};
\node at (5,5.7) {$A$};
\node at (6,6.3) {$C^\pr$};
\node at (6,5.7) {$B^\pr$};

\draw [thick] (2,5) -- (5,6) -- (2,7);
\draw (2,5) -- (5.72,6.24);
\draw (2,7) -- (5.72,5.76);

\draw [thick](7,5) arc (210:150:2);
\draw (7,5) arc (210:580:2);

\node [vertex] at (7,5) {};
\node [vertex] at (7,7) {};
\node [vertex] at (10.73,6) {};

\node at (6.7,5) {$C$};
\node at (6.7,7) {$B$};
\node at (11,6) {$A$};

\draw [thick] (7,5) -- (10.73,6) -- (7,7);

\draw [thick](12,5) arc (210:150:2);
\draw (12,5) arc (210:580:2);

\node [vertex] at (12,5) {};
\node [vertex] at (12,7) {};
\node [vertex] at (16.5,6) {};

\draw [thick] (12,5) -- (16.5,6) -- (12,7);
\node [vertex] at (15.72,6.18) {};
\node [vertex] at (15.72,5.82) {};

\node at (11.7,5) {$C$};
\node at (11.7,7) {$B$};
\node at (16.8,6) {$A$};
\node at (16,6.4) {$B^\pr$};
\node at (16,5.6) {$C^\pr$};

\node at (3.7,3.5) {$(a)$};
\node at (8.7,3.5) {$(b)$};
\node at (13.7,3.5) {$(c)$};
\end{tikzpicture}
\]
\caption{Various cases of relative position of triangle $\triag ABC$ and the cycle $\ca$.}
\label{fig:proper}
\end{figure}

Therefore, hyperbolic, Euclidean, and spherical triangles do not exhaust all M\"obius triangles; the remaining ones are \textit{pure M\"obius} triangles (Figure~\ref{fig:pure}).

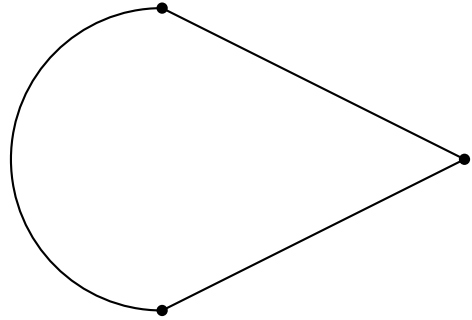
\begin{figure}[h]
\[
\begin{tikzpicture}
\tikzstyle{vertex}=[circle, fill, inner sep=1.5pt]
\draw [thick](6,0) arc (270:90:2);
\node [vertex] at (6,0) {};
\node [vertex] at (6,4) {};
\draw [thick](6, 0) -- (10, 2) -- (6, 4);
\node [vertex] at (10,2) {};
\end{tikzpicture}
\]
\caption{A pure M\"obius triangle.}
\label{fig:pure}
\end{figure}

\section{Trilinear coordinates}
In our most general considerations we will not need triangles. A \textit{generalized triangle} $\Delta$ is an ordered triple $(\ca,\cb,\cc)$ of non-collinear oriented cycles. For $\ca,\cb,\cc\approx L, M, N$ let
$$
\GG(\ca,\cb,\cc) = \{\cn\sim uL+vM+wN|u, v, w\in\R\}
$$
be the two-dimensional plane in $\GG$ generated by $\ca,\cb,\cc$. The numbers $u,v,w$ are the \textit{trilinear coordinates} of the cycle $\cn$ if $\cn\sim uL+vM+wN$. Clearly, $u,v,w$ are defined up to a common nonzero factor, i.e. these are homogeneous coordinates on $\GG(\ca,\cb,\cc)$.

We use notation $\cn=[u:v:w]$. For example, we have $\ca=[1:0:0]$, or, more generally, $\cn=[1:-\lambda:0]$ when $(\ca,\cb;\cn)=\lambda$.
The trilinear coordinates establish a homeomorphism between $\GG(\ca,\cb,\cc)$ and the projective plane $\R P^2$.

Next we introduce pencil planes. Let 
$$
\PP(\ca,\cb,\cc)=\{\gP\in \PP|
\gP\subset\GG(\ca,\cb,\cc)
\}
$$
be the collection of pencils containing the cycles of $\GG(\ca,\cb,\cc)$ only.
Using the projective duality we introduce trilinear coordinates on $\PP(\ca,\cb,\cc)$ as well. We say that $\gP$ has coordinates $x,y,z$ when $ux+vy+wz=0$ for every cycle $\ca=[u:v:w]\in\gP$. We use notation $\gP=(x:y:z)$. For example, $\GG(\cb,\cc)=(1:0:0)$. The coordinates similarly establish a homeomorphism between $\PP(\ca,\cb,\cc)$ and $\R P^2$.

We list then the standard projective coordinate properties.

\begin{prop}
Three cycles $\cn_1=[u_1:v_1:w_1]$, $\cn_2=[u_2:v_2:w_2]$, $\cn_3=[u_3:w_3:w_3]$ 
are collinear iff
$$
\left|
\begin{array}{ccc}
u_1 & v_1 & w_1 \\
u_2 & v_2 & w_2 \\
u_3 & v_3 & w_3 
\end{array}
\right|=0.
$$
\end{prop}

\begin{prop}
Three pencils $\gP_1=(x_1:y_1:z_1)$, $\gP_2=(x_2:y_2:z_2)$, $\gP_3=(x_3:y_3:z_3)$ 
are concurrent iff
$$
\left|
\begin{array}{ccc}
x_1 & y_1 & z_1 \\
x_2 & y_2 & z_2 \\
x_3 & y_3 & z_3 
\end{array}
\right|=0.
$$
\end{prop}

\begin{prop}
Let $\cn_1=[u_1:v_1:w_1]$, $\cn_2=[u_2:v_2:w_2]$ then 
$$
\GG(\cn_1,\cn_2)=(v_1w_2-v_2 w_1: u_2w_1-u_1w_2: u_1v_2-u_2v_1).
$$
\end{prop}

\begin{prop}
Let $\gP_1=(x_1:y_1:z_1)$, $\gP_2=(x_2:y_2:z_2)$ be two distinct pencils. Then $\gP_1\cap\gP_2$ contains the only cycle
$\cn=[y_1z_2-y_2 z_1: x_2z_1-x_1z_2: x_1y_2-x_2y_1]$.
\end{prop}

Non-collinear cycles $\ca,\cb,\cc$ generate a three-dimensional subspace $V\subset\Herm(2)$. For every pencil $\gP$ generating a subspace $W$ we have either $W\subset V$ and this happens iff $\gP\subset\GG(\ca,\cb,\cc)$, or
$\dim(V\cap W)=1$, which means that $\gP$ intersects $\GG(\ca,\cb,\cc)$ in exactly one cycle. This applies in particular to all pencils $\gP^\perp$, where $\gP\in\PP(\ca,\cb,\cc)$ is non-parabolic.

For the planes $\GG(\ca,\cb,\cc)$ we also have types depending on the signature of the space $V$. Only three cases are possible.
\begin{itemize}
\item[$(i)$] If $\sig V=(3,0)$ then $V$ is \textit{elliptic} or \textit{spherical}. All pencils and all cycles are also elliptic.  
\item[$(ii)$] If $\sig V=(2,0)$ then $V$ is \textit{parabolic} or \textit{Euclidean}. All pencils and cycles are elliptic or parabolic. In this case the inner product is degenerate on $V$. We later identify all such cases. 
\item[$(iii)$] If $\sig V=(2,1)$ then $V$ is \textit{hyperbolic}. In this case all types of cycles and pencils appear. 
\end{itemize}

The type of the plane is simply recovered from the Gram matrix $\Gamma$.
\begin{prop}
Let $$\Gamma=\left(
\begin{matrix}
1 & \la\ca,\cb\ra & \la\ca,\cc\ra \\
\la\ca,\cb\ra & 1 & \la\cb,\cc\ra \\
\la\ca,\cc\ra & \la\cb,\cc\ra & 1\\
\end{matrix}
\right)$$
 be the Gram matrix for three non-collinear cycles $\ca,\cb,\cc$.
Then the corresponding hyperplane $V\subset\Herm(2)$ is elliptic iff $\det\Gamma>0$, parabolic iff $\det\Gamma=0$, and hyperbolic iff $\det\Gamma<0$. 
\end{prop}
\begin{proof}
Indeed, the signature $\sig V$ gives the numbers of positive and negative eigenvalues of the matrix $\Gamma$. There are exactly three possible cases that give the different signs of the determinant.
\end{proof}

When $\Delta$ is a standard triangle trilinear coordinates agree with the
existing well-known concept.
\begin{itemize}
\item in the spherical case any three non-collinear oriented lines uniquely define a triangle that lies on the left for all of them. All generalized triangles are ordinary, and nothing new is added.

\item in the Euclidean case two sides can be parallel, but not all three. This means that one of the vertices (and only one) of a generalized triangle can lie  at infinity, and the corresponding inner angle will be zero. Trilinear coordinates are still well-defined.

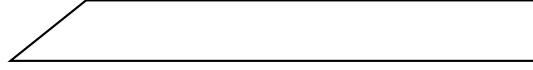
\begin{figure}[h]
\[
\begin{tikzpicture}
\draw [thick](8,2) -- (2,2) -- (1, 1.2) -- (8, 1.2) {};
\end{tikzpicture}
\]
\caption{A generalized Euclidean triangle with a zero angle.}
\label{fig:generalized-euclid}
\end{figure}

\item in hyperbolic geometry everything, as usual, is more complicated. Generalized triangles may have no vertices at all. Moreover, we cannot in general identify a triangle with some part of the plane bounded by the lines, as one of the lines may split the other two (Figure~\ref{fig:generalized-hyperbolic}).

\begin{figure}[h]
\[
\begin{tikzpicture}
\draw (2,0) -- (17,0) {};
\draw [thick] (3, 0) arc (180: 0: 3.5);
\draw [thick] (4, 0) arc (180: 0: 2.3);
\draw [thick] (6, 0) arc (180: 0: 1);
\draw [thick] (12,0) arc (180:0: 2);
\draw [thick] (12.5,0) arc (180:0: 1);
\draw [thick] (13.5,0) arc (180:0: 1);
\end{tikzpicture}
\]
\caption{Generalized hyperbolic triangles.}
\label{fig:generalized-hyperbolic}
\end{figure}
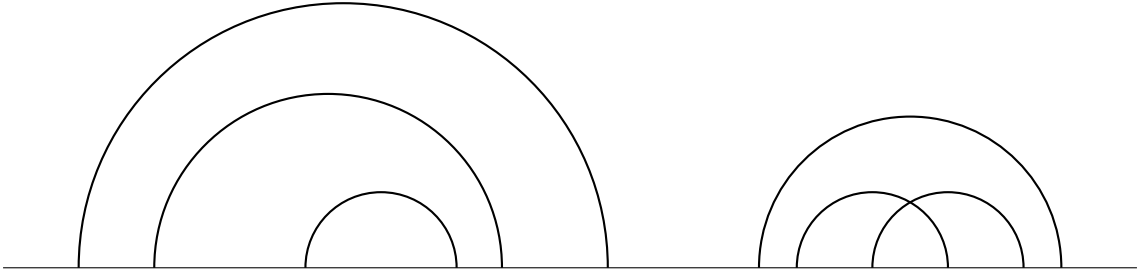
\end{itemize}

\section{Ceva's Theorem}
Every triangle defines three digons or monogons by selecting two of the three triangle's bounding cycles. Thus, we can consider cevians of these digons -- cevians of the triangle. Let $\gA=\GG(\cb,\cc)$, $\gB=\GG(\ca,\cc)$, $\gC=\GG(\ca,\cb)$. Consider three arbitrary cevians $\cn_a\in\gA$, $\cn_b\in\gB$, 
$\cn_c\in\gC$.
In Euclidean case the Ceva's Theorem in trigonometric form gives a criterion when $\cn_a,\cn_b,\cn_c$ intersect in a single point, or are parallel \cite{Cox}. In other words, when the conditions of the theorem are met 
$\cn_a,\cn_b,\cn_c$ are collinear, i.e. belong to a single pencil. This is exactly what we achieve in the M\"obius case. The equivalent of the Ceva's Theorem is the following.
\begin{theorem}
Let $\Delta=\triag ABC$ be a proper triangle with positive angles. 
Let $\cn_a, \cn_b, \cn_c$ be its cevians that are distinct from the sides of the triangle. Then $\cn_a, \cn_b, \cn_c$ are collinear iff
$$
\frac{\sin\angle(\cb, \cn_a)}{\sin\angle(\cc, \cn_a)}
\frac{\sin\angle(\cc, \cn_b)}{\sin\angle(\ca, \cn_b)}
\frac{\sin\angle(\ca, \cn_c)}{\sin\angle(\cb, \cn_c)}
=1.
$$
\end{theorem}

Unfortunately, we need some extra conditions of positivity of the angles since otherwise we will encounter division by zero. However, we can rewrite the formulas using the splitting factors introduced earlier. But then we do not need the sides to intersect. Thus, in the most general formulation of the Ceva's Theorem we don't need a triangle at all. We only need three non-collinear oriented cycles and three cevians from the corresponding pencils. We prove the theorem in this most general form.
\begin{theorem}[Ceva's Theorem]
Let $\ca, \cb,\cc$ be three distinct non-collinear cycles. 
Let $\cn_a\in\gA$, $\cn_b\in\gB$, $\cn_c\in\gC$ be three cycles distinct from $\ca$, $\cb$, $\cc$. Then $\cn_a$, $\cn_b$, $\cn_c$ are collinear iff the following identity holds
$$
(\cb,\cc;\cn_a)(\cc,\ca;\cn_b)(\ca,\cb;\cn_c) = 1.
$$
Moreover, if $\cn_a$, $\cn_b$ are elliptic, then
$$
(\cn_a,\cn_b;\cn_c) = (\cn_a,\cc;\cb)(\cc,\cn_b;\ca).
$$
\end{theorem}

\begin{proof}
Let $\lambda=(\cb,\cc;\cn_a)$, $\mu = (\cc,\ca;\cn_b)$, $\nu = (\ca,\cb;\cn_c)$. Then $\cn_a=[0:1:-\lambda]$, $\cn_b=[-\mu:0:1]$, $\cn_c=[1:-\nu:0]$. Therefore, $\cn_a, \cn_b,\cn_c$ are collinear iff
$$
\left|
\begin{array}{ccc}
0 & 1 & -\lambda\\
-\mu & 0 & 1\\
1 & -\nu & 0
\end{array}
\right| = 1-\lambda\mu\nu=0.
$$ 
Assume now $\lambda\mu\nu=1$. Let $\ca,\cb,\cc\approx L, M, N$ respectively.
Then $\cn_a\approx (\cn_a,\cb;\cc)(M-\lambda N)$, 
$\cn_b\approx(\cn_b,\cc;\ca)(N-\mu L)$, and $\cn_c\sim L-\nu M = -\nu(M-\lambda N) -\lambda \nu (N-\mu L)$. Therefore,
$$
(\cn_a,\cn_b;\cn_c)=-\lambda (\cn_a,\cb;\cc)(\cn_b,\cc;\ca)^{-1}=
(\cn_a,\cc;\cb)(\cc,\cn_b;\ca).
$$
\end{proof}

When the conditions of Ceva's Theorem are met $\cn_a,\cn_b,\cn_c$ define a pencil $\gP=(x:y:z)$, where 
$$
\frac{y}{z} = (\cb,\cc;\cn_a),\quad
\frac{z}{x} = (\cc,\ca;\cn_b),\quad
\frac{x}{y} = (\ca,\cb;\cn_c).
$$
Next we identify the type of $\gP$. 
\begin{theorem}\label{type1}
Let $\gP=(x:y:z)$, let 
\begin{gather*}
X = x^2(1-\la\cb,\cc\ra^2) + y^2(1 - \la\ca,\cc \ra^2) + z^2(1-\la\ca,\cb\ra^2)+\\
+2xy(\la\ca,\cc\ra\la\cb,\cc\ra-\la\ca,\cb\ra )
+2yz(\la\ca,\cb\ra\la\ca,\cc\ra-\la\cb,\cc\ra )
+2xz(\la\ca,\cb\ra\la\cb,\cc\ra-\la\ca,\cc\ra ).
\end{gather*}
Then $\gP$ is elliptic iff $X>0$, parabolic iff $X=0$, and hyperbolic iff $X<0$.
\end{theorem}

\begin{proof}
Let $\ca,\cb,\cc\approx L, M, N$.
The pencil $\gP$ corresponds to the subspace $W$ generated by $M-\lambda N$, $N-\mu L$. Therefore, according to Proposition~\ref{W-type} we simply need to compute the determinant of the Gram matrix
\begin{gather*}
\left|
\begin{array}{cc}
1-2\lambda\la\cb,\cc\ra+\lambda^2 & -\lambda+\la\cb,\cc\ra +  \lambda\mu\la\ca,\cc\ra -\mu\la\ca,\cb\ra \\
-\lambda+\la\cb,\cc\ra +  \lambda\mu\la\ca,\cc\ra -\mu\la\ca,\cb\ra &
1-2\mu\la\ca,\cc\ra+\mu^2
\end{array}
\right|=\\
=(1-2\lambda\la\cb,\cc\ra+\lambda^2)(1-2\mu\la\ca,\cc\ra+\mu^2) - 
(-\lambda+\la\cb,\cc\ra +  \lambda\mu\la\ca,\cc\ra -\mu\la\ca,\cb\ra)^2=\\
=1+\lambda^2\mu^2+\mu^2-\la\cb,\cc\ra^2 - \lambda^2\mu^2\la\ca,\cc\ra-\mu^2\la\ca,\cb\ra^2+2\lambda\mu\la\ca,\cc\ra\la\cb,\cc\ra+\\
+2\mu\la\ca,\cb\ra\la\cb,\cc\ra
+2\lambda\mu^2\la\ca,\cb\ra\la\ca,\cc\ra
-2\lambda\mu\la\ca,\cb\ra-2\lambda\mu^2\la\cb,\cc\ra-2\mu\la\ca,\cc\ra.
\end{gather*}
Plug the values for $\lambda$ and $\mu$. Cases when $\lambda$ or $\mu$ are equal to $\infty$ need some special attention that we omit. 
\end{proof}
When $\ca,\cb,\cc$ form a triangle we have $\la\ca,\cb\ra=-\cos\gamma$,
$\la\ca,\cc\ra=-\cos\beta$, $\la\cb,\cc\ra=-\cos\alpha$ and the expression becomes
\begin{gather*}
X = x^2\sin^2\alpha+y^2\sin^2\beta+z^2\sin^2\gamma+
2xy(\cos\alpha\cos\beta+\cos\gamma)+\\
2yz(\cos\alpha\cos\gamma+\cos\beta)+
2xz(\cos\beta\cos\gamma+\cos\alpha)=\\
=(x\sin\alpha+y\sin\beta+z\sin\gamma)^2+
2xy\big(\cos(\alpha+\beta)+\cos\gamma\big)+\\
2yz\big(\cos(\alpha+\gamma)+\cos\beta\big)+
2xz\big(\cos(\beta+\gamma)+\cos\alpha\big).
\end{gather*}
For an Euclidean triangle we have $\cos\gamma +\cos(\alpha+\beta)=0$ and etc. Therefore, in this case
$$
X = (x\sin\alpha+y\sin\beta+z\sin\gamma)^2\ge 0.
$$
This implies what we already know that all linear pencils on the Euclidean plane are elliptic or parabolic, and the parabolic lie on the line
$$
x\sin\alpha+y\sin\beta+z\sin\gamma=0.
$$ 
This is the line at infinity $[\sin\alpha:\sin\beta:\sin\gamma]$.
\medskip

Finally, there is one more piece of information we would like to recover. If $\gP$ is parabolic or hyperbolic then exactly one of the cevians $\cn_a, \cn_b,\cn_c$ is the splitting one. We identify which one.

\begin{theorem}\label{type2}
Assume the conditions of the Ceva's Theorem are met, and $\cn_a, \cn_b,\cn_c$ are elliptic and belong to a unique pencil $\gP=(x:y:z)$ with $xyz\neq 0$. Let 
$$
a=\left|\frac{1}{y}(\cb,\cn_a;\cc)\right|,\quad
b=\left|\frac{1}{z}(\cc,\cn_b;\ca)\right|,\quad
c=\left|\frac{1}{x}(\ca,\cn_c;\cb)\right|.
$$
Then $\gP$ is elliptic iff $a,b,c$ satisfy the triangle inequalities. It is parabolic if one of the numbers is the sum of the remaining two (and the greatest number corresponds to the splitting cevian), and hyperbolic iff one of the numbers is greater than the sum of the others (and it corresponds to the splitting cevian again).
\end{theorem}

\begin{proof}
The numbers $a,b,c$ have the property that
\begin{align*}
|(\cn_a,\cn_b;\cn_c)| = \frac{b}{a},\quad
|(\cn_b,\cn_c;\cn_a)| = \frac{c}{b},\quad
|(\cn_c,\cn_a;\cn_b)| = \frac{a}{c}.
\end{align*}
Indeed, we have that
$$
|(\cn_a,\cn_b;\cn_c)| = |(\cn_a,\cc;\cb)(\cc,\cn_b;\ca)|.
$$
On the other hand,
$$
\frac{b}{a} = \left|\frac{y}{z}\frac{(\cc,\cn_b;\ca)}{(\cb,\cn_a;\cc)}
\right| = |(\cb,\cc;\cn_a)(\cn_a,\cb;\cc) (\cc,\cn_b;\ca)|
=|(\cn_a,\cc;\cb) (\cc,\cn_b;\ca)|.
$$
And similarly for the other two expressions. Now consider the cases

If $\gP$ is elliptic then $a, b, c$ are the values $|\sin\angle(\cn_b,\cn_c)|$, 
$|\sin\angle(\cn_a,\cn_c)|$, $|\sin\angle(\cn_a,\cn_b)|$ up to a positive constant. We have
\begin{gather*}
|\sin\angle(\cn_a,\cn_c)| = |\sin(\angle(\cn_a,\cn_b)+\angle(\cn_b,\cn_c))|
\le\\
 |\sin\angle(\cn_a,\cn_b)\cos\angle(\cn_b,\cn_c)| + 
|\sin\angle(\cn_b,\cn_c)\cos\angle(\cn_a,\cn_b)|<\\
<|\sin\angle(\cn_a,\cn_b)| + |\sin\angle(\cn_b,\cn_c)|.
\end{gather*}
This proves $b<a+c$. Absolutely identically we show that $c<a+b$ and $a<b+c$.

Assume now $\gP$ is parabolic. Then $a,b,c$ are the distances between $\cn_a$, $\cn_b$, $\cn_c$ when presented as parallel straight lines, up to a positive constant.
Clearly, $\cn_a$ splits the other two cevians iff $a=b+c$.

Finally, assume $\gP$ is hyperbolic. Then $a, b, c$ are the values 
$\sinh 2\pi|\mu(\cn_b,\cn_c)|$, $\sinh 2\pi|\mu(\cn_a,\cn_c)|$, 
$\sinh 2\pi|\mu(\cn_a,\cn_b)|$ up to a positive constant.
Assume $\cn_b$ is the splitting cevian. Then 
$|\mu(\cn_a,\cn_c)|=|\mu(\cn_a,\cn_b)|+|\mu(\cn_b,\cn_c)|$ and
\begin{gather*}
\sinh2\pi|\mu(\cn_a,\cn_c)| = \sinh2\pi(|\mu(\cn_a,\cn_b)|+|\mu(\cn_b,\cn_c)|)
\ge\\
\sinh2\pi|\mu(\cn_a,\cn_b)|\cosh2\pi|\mu(\cn_b,\cn_c)| + 
\sinh2\pi|\mu(\cn_b,\cn_c)|\cosh2\pi|\mu(\cn_a,\cn_b)|>\\
>\sinh2\pi|\mu(\cn_a,\cn_b)| + \sinh2\pi|\mu(\cn_b,\cn_c)|.
\end{gather*}
This implies $b>a+c$.
\end{proof}
\begin{corollary} \label{type3}
For a proper triangle with angles $\alpha,\beta,\gamma$ we have
$$
a=\frac{\sqrt{y^2 + 2yz\cos\alpha+z^2}}{|yz|},\quad 
b=\frac{\sqrt{x^2 + 2xz\cos\beta+z^2}}{|xz|},\quad
c=\frac{\sqrt{x^2 + 2xy\cos\gamma+y^2}}{|xy|}.
$$
\end{corollary}
\begin{proof}
Let $\lambda=(\cb,\cc;\cn_a)=\frac{y}{z}$. Then according to Corollary~\ref{acb} 
$$
|(\cb,\cn_a;\cc)|=
\sqrt{1+2\lambda\cos\alpha+\lambda^2}.
$$
Therefore,
$$
\left|\frac{1}{y}(\cb,\cn_a;\cc)\right|=
\left|
\frac{\sqrt{y^2 + 2yz\cos\alpha+z^2}}{yz}
\right|.
$$

\end{proof}

\begin{prop}
Let $\ca,\cb,\cc$ be three non-collinear cycles. Let $\cn_a\in\CC(\cb,\cc)$,
$\cn_b\in\CC(\ca,\cc)$, $\cn_c\in\CC(\cb,\cc)$ distinct from $\ca,\cb,\cc$. Assume that $\cn_a,\cn_b,\cn_c$
have a common point. Then they are collinear.
\end{prop} 


\begin{proof}
Let $\infty$ be a common point of $\cn_a$, $\cn_b$, $\cn_c$. Then $\cn_a$, $\cn_b$, $\cn_c$ are straight lines. On the other hand, $\cn_a$, $\cn_b$, $\cn_c$ are the radical axes of every pair of $\ca,\cb,\cc$ -- so they must either intersect in a common point, or be parallel. In both cases $\cn_a, \cn_b, \cn_c$ are collinear.
\end{proof}

Consider now the standard geometries.
Assume we have a triangle $\triag ABC$ and three cevians $\cn_a, \cn_b, \cn_c$ that satisfy the condition on the Ceva's Theorem. Then
\begin{itemize}
\item for a spherical triangle $\cn_a, \cn_b, \cn_c$ intersect in two antipodal points that are uniquely defined.

\item for a Euclidean triangle $\cn_a, \cn_b, \cn_c$ either intersect in a point or are parallel.

\item for a hyperbolic triangle we have that either $\cn_a, \cn_b, \cn_c$ intersect in a point, are parallel (that is, approach the same point at infinity), or are pairwise disjoint and share a unique common perpendicular.
\end{itemize}

\section{Bisectors, Incenter and Excenters}
Consider the first example. For any proper triangle consider the bisectors $\cl_a, \cl_b, \cl_c$ of its angles. Unlike the Euclidean case things become more complicated as the bisectors may not intersect. However, the conditions of Ceva's Theorem are always met and the \textit{incenter} $\gI=(1:1:1)$ always exists, but incenter  \textit{is a pencil containing the bisectors}, not a point. And as a pencil it can have any type in general (Figure~\ref{fig:incenter}).

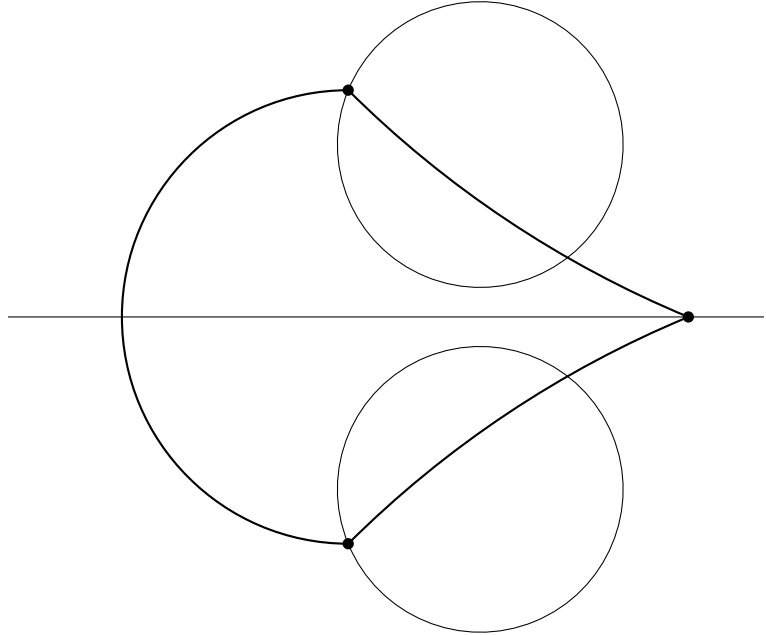
\begin{figure}[h]
\[
\begin{tikzpicture}
\tikzstyle{vertex}=[circle, fill, inner sep=1.5pt]
\draw [thick](4.500000,1.500000) arc (269.880135:90.119865:3.000007);
\node [vertex] at (4.500000,1.500000) {};
\node [vertex] at (4.500000,7.500000) {};
\draw [thick](9.000000,4.500000) arc (-112.500000:-134.880135:13.934369);
\draw [thick](4.500000,1.500000) arc (134.880135:112.500000:13.934369);
\node [vertex] at (9.000000,4.500000) {};
\draw (4.500000,1.500000) arc (-157.619865:202.380135:1.889339);
\draw (4.500000,7.500000) arc (157.619865:517.619865:1.889339);
\draw (0, 4.5) -- (10, 4.5);
\end{tikzpicture}
\]
\caption{An isosceles M\"obius triangle with angles $\frac{\pi}{4}$, $\frac{3\pi}{4}$, $\frac{3\pi}{4}$ and a hyperbolic incenter.}
\label{fig:incenter}
\end{figure}

\begin{theorem}
Let $\alpha\leq\beta\leq\gamma$ be the angles of a triangle. Then the incenter $\gI$ is hyperbolic, and $\cl_a$ splits $\cl_b$, $\cl_c$ iff 
$\cos\frac{\alpha}{2}>\cos\frac{\beta}{2}+\cos\frac{\gamma}{2}$.
\end{theorem}
\begin{proof}
This follows from Corollary~\ref{type3} after plugging $x=y=z=1$.
\end{proof}

However, for all standard geometries $\gI$ is elliptic.
Moreover, the \textit{excenters} $\gI_a, \gI_b, \gI_c$ are also well-defined and are also pencils; e.g. $\gI_a$ is the unique pencil containing $\cl_a, \cl_b^\perp, \cl_c^\perp$; furthermore, $\gI_a=(1:-1:-1)$. 

\begin{theorem}
The excenter $\gI_a$ is hyperbolic iff one of the following cases is true:
\begin{enumerate}
\item[$(i)$]
$\cos\frac{\alpha}{2}>\sin\frac{\beta}{2}+\sin\frac{\gamma}{2}$; 
 in this case $\cl_a$ splits $\cl_b^\perp$ and $\cl_c^\perp$.
\item[$(ii)$]
$\sin\frac{\beta}{2}>\cos\frac{\alpha}{2}+\sin\frac{\gamma}{2}$;
 in this case $\cl_b^\perp$ splits $\cl_a$ and $\cl_c^\perp$.
\item[$(iii)$]
$\sin\frac{\gamma}{2}>\cos\frac{\alpha}{2}+\sin\frac{\beta}{2}$;
 in this case $\cl_c^\perp$ splits $\cl_a$ and $\cl_b^\perp$.
\end{enumerate}
\end{theorem}
\begin{proof}
This also follows from Corollary~\ref{type3}.
\end{proof}

For a hyperbolic triangle this has some interesting consequences. It is impossible for $\cl_b^\perp$ or $\cl_c^\perp$ to split $\gI_a$. Indeed, we have that
$$
\sin\frac{\beta}{2}=
\cos\frac{\pi-\beta}{2}
 <\cos\frac{\alpha+\gamma}{2}=
\cos\frac{\alpha}{2}\cos\frac{\gamma}{2} - \sin\frac{\alpha}{2}\sin\frac{\gamma}{2}\le\cos\frac{\alpha}{2}+\sin\frac{\gamma}{2}.
$$
Therefore, case $(ii)$ is impossible. Similarly, $(iii)$ is impossible as well.
But the excenter $\gI_a$ can still be hyperbolic (Figure~\ref{fig:excenter}) and its type defines the type of the corresponding excircle, i.e. whether it is a circle, a horocycle, or an equidistant curve. 

\begin{figure}[h]
\[
\begin{tikzpicture}
\tikzstyle{vertex}=[circle, fill, inner sep=1.5pt]
\draw [thick](3.435000,2.580000) arc (100.027395:41.614589:2.620022);
\node [vertex] at (3.435000,2.580000) {};
\node [vertex] at (5.850000,1.740000) {};
\draw [thick](5.850000,1.740000) arc (144.400539:79.600858:2.989098);
\node [vertex] at (8.820000,2.940000) {};
\draw [thick](3.435000,2.580000) arc (138.115402:49.533933:3.864402);

\draw (5.291356,0.000000) arc (180.000000:0.000000:2.989098);
\draw (2.447611,0.000000) arc (180.000000:0.000000:3.864402);
\draw (1.271174,0.000000) arc (180.000000:0.000000:2.620022);

\draw [thick](4.631158,0.000000) arc (0.000000:180.000000:0.503400);
\draw (3.624358,0.000000) arc (360-111.933320:-68.066680:1.347692);

\draw [thick](0.985739,0.000000) arc (180.000000:0.000000:5.336572);
\draw (11.658882,0.000000) arc (18.180299:161.819701:5.616974);

\draw [thick](9.246719,0.000000) arc (0.000000:180.000000:1.193393);
\draw (6.859932,0.000000) arc (360-133.801587:-46.198413:1.724150);

\draw (0, 0) -- (12.5,0);
\end{tikzpicture}
\]
\caption{A hyperbolic triangle with all excircles being equidistant curves. The excenters in this case are lines, and this makes perfect sense as an excircle remains the set of points equidistant from the corresponding excenter.}
\label{fig:excenter}
\end{figure}
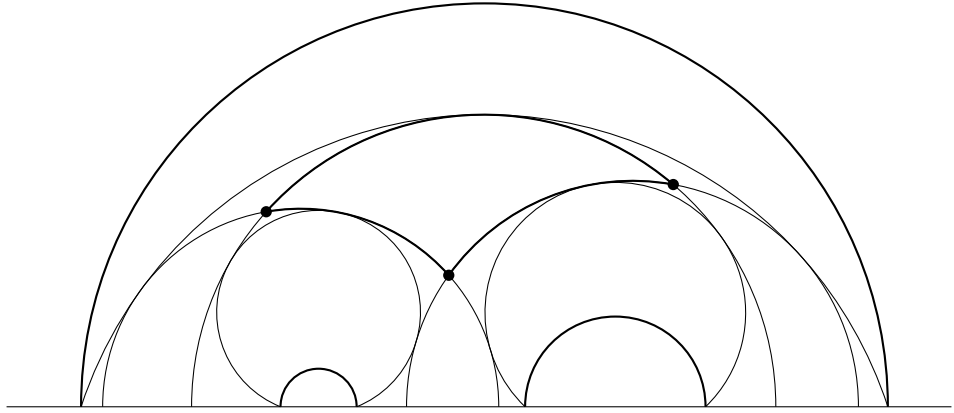

\begin{theorem}[\cite{Bib}]\label{excircle}
Let $\alpha,\beta,\gamma$ be the angles of a hyperbolic triangle.
Let $\ce_a$ be the excircle tangent to the side $\ca$ from outside. Then
\begin{enumerate}
\item[$(i)$] if $\cos\frac{\alpha}{2} < \sin\frac{\beta}{2}+\sin\frac{\gamma}{2}$
then $\ce_a$ is a hyperbolic circle.
\item[$(ii)$] if $\cos\frac{\alpha}{2} = \sin\frac{\beta}{2}+\sin\frac{\gamma}{2}$
then $\ce_a$ is a horocycle.
\item[$(iii)$] if $\cos\frac{\alpha}{2} > \sin\frac{\beta}{2}+\sin\frac{\gamma}{2}$
then $\ce_a$ is an equidistant curve.
\end{enumerate}
\end{theorem}
We need an additional result first.

\begin{lemma}\label{tang}
Let $\ca$ be a cycle tangent to $\cb$ and $\cc$ from the left. Then $\ca\perp\cl$, where $\cl$ is the bisector of the digon, monogon, or annulus bounded by $\cb$ and $\cc$.
\end{lemma}
\begin{proof}
Let $\ca,\cb,\cc\approx L, M, N$. Then $\cl\sim M-N$, and $\la L,M\ra=\la L,N\ra=\pm1$. Then $\la L,M-N\ra=0$ implying $\ca\perp\cl$.
\end{proof}

\begin{proof}[Proof of Theorem~\ref{excircle}]
The excircle $\ce_a$ is tangent to $\cb,\cc$ from the left and to $\ca$ from the right. Therefore, Lemma~\ref{tang} implies $\ce_a\perp\cl_a$. Analogously, we show that $\ce_a\perp\cl_b^\perp$ and $\ce_a\perp\cl_c^\perp$. Therefore, $\ce_a\in\gI_a^\perp$.
Hence, the type of $\gI_a$ defines whether $\ce_a$ is a circle, horocycle, or an equidistant curve.
\end{proof}

\section{Altitudes and Orthocenter}
An \textit{altitude} of a M\"obius triangle is a unique cevian that is orthogonal to the opposite triangle side. We will see that even for proper triangles altitudes may not exist in general.

\begin{prop}\label{ca_perp_gP}
Let $\gP$ be an extended pencil and $\cc$ be a cycle. If $\cc\not\in\gP^\perp$ then there exists a unique $\cn\in\gP$ such that $\cn\perp\cc$,
\end{prop}
\begin{proof}
Let $\gP=\GG(\ca,\cb)$  for some $\ca,\cb\sim L, M$. Let $\cc\sim N$. Then unless $\ca\perp\cc$ and $\cb\perp\cc$, which is equivalent to $\cc\in\gP^\perp$, the equation $\la L-\lambda M,N\ra =0$ has a unique solution for $\lambda\in\oline{\R}$ (with $\lambda=\infty$ iff $\la M,N\ra=0$) which uniquely corresponds to $\cn\in\gP$ satisfying the condition.
\end{proof}

Consider a triangle $\triag ABC$. Let $\ch_a$, $\ch_b$, $\ch_c$ be the altitudes passing through $A$, $B$, $C$ respectively. The \textit{orthocenter} $\gH$ of a triangle is a unique pencil containing its altitudes. We investigate now when the altitudes and the orthocenter exist.

\begin{prop}
The altitude $\ch_a$ exists unless $\beta=\gamma=\frac{\pi}{2}$.
\end{prop}
\begin{proof}
According to the definition and Proposition~\ref{ca_perp_gP} $\ch_a$ exists unless $\ca\in\gA^\perp$. Since $\gA=\GG(\cb,\cc)$ the latter is equivalent to $\ca\perp\cb$ and $\ca\perp\cc$ holding simultaneously.
\end{proof}
 
Consider now a cevian $\cn\in\gA$. What is the angle 
$\angle(\ca,\cn)$?

\begin{theorem}
Let $\ca,\cb,\cc$ be three non-collinear oriented cycles.
Let $\cn\in\CC(\cb,\cc)$. Then
$$
\la\ca,\cn\ra = (\cn,\cb;\cc)\la\ca,\cb\ra + (\cn,\cc;\cb)\la\ca,\cc\ra.
$$
\end{theorem}
\begin{proof}
Let $\ca,\cb,\cc\approx L,M, N$. Let $\lambda = (\cb,\cc;\cn)$. Then 
$\cn\approx (\cn,\cb;\cc)(M-\lambda N)$. Therefore
$$
\la\ca,\cn\ra = (\cn,\cb;\cc)\la L, M-\lambda N\ra = 
 (\cn,\cb;\cc)\la \ca,\cb\ra - \lambda (\cn,\cb;\cc)\la\ca,\cc \ra.
$$
On the other hand, $- \lambda (\cn,\cb;\cc) = (\cc,\cn;\cb)^{-1}=(\cn,\cc;\cb)$. 
\end{proof}

\begin{corollary}
If the altitude $\ch_a$ exists then
$$
(\cb,\cc;\ch_a) = \frac{\la\ca,\cb\ra}{\la\ca,\cc\ra} = \frac{\cos\gamma}{\cos\beta}.
$$
\end{corollary}
\begin{proof}
Since $\ch_a\perp\ca$ we have that
$$
0=\la\ca,\ch_a\ra = 
(\ch_a,\cb;\cc)\la\ca,\cb\ra + (\ch_a,\cc;\cb)\la\ca,\cc\ra
$$ 
implying that
$$
\frac{\la\ca,\cb\ra}{\la\ca,\cc\ra} = 
-\frac{(\ch_a,\cc;\cb)}{(\ch_a,\cb;\cc)}=
-\frac{1}{(\cc,\ch_a;\cb)(\ch_a,\cb;\cc)}=
(\cb,\cc;\ch_a)
$$
according to Proposition~\ref{splitting-prop}.
\end{proof}
Hence, from the Ceva's Theorem it follows that the orthocenter $\gH$ exists unless at least two angles are right. If, for example, $\alpha=\beta=\frac{\pi}{2}\neq\gamma$ then $\ch_a=\ch_b=\cc$ and $\ch_c$ does not
exist; in this case we have only one altitude and the orthocenter is undefined.
This may occur to a spherical triangle. When all angles are right then the triangle has no altitudes at all.

On the other hand, for Euclidean (except the generalized triangle with angles $0$, $\frac{\pi}{2}$, $\frac{\pi}{2}$) or hyperbolic triangles the orthocenter always exists. However, for a hyperbolic triangle it may be a point at infinity, or a line (Figure~\ref{fig:orthocenter}). 

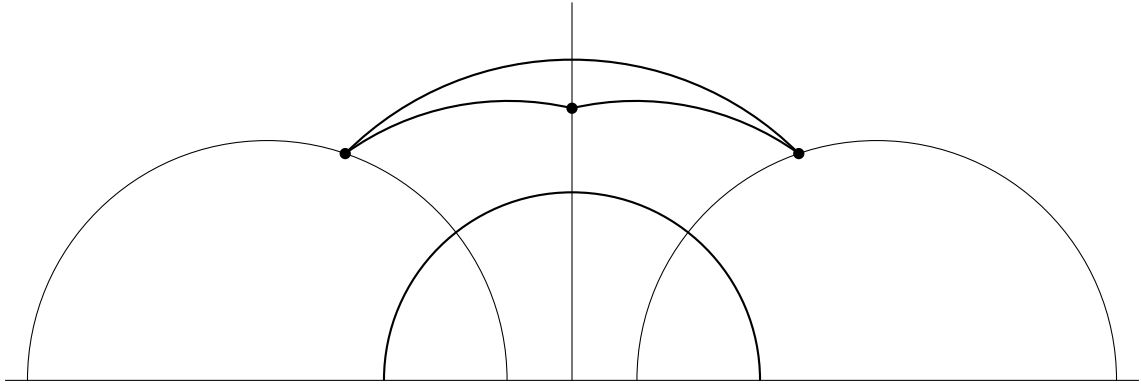
\begin{figure}[h]
\[
\begin{tikzpicture}
\tikzstyle{vertex}=[circle, fill, inner sep=1.5pt]
\draw [thick](3.000000,3.000000) arc (125.753887:76.865978:3.696701);
\node [vertex] at (3.000000,3.000000) {};
\node [vertex] at (6.000000,3.600000) {};
\draw [thick](6.000000,3.600000) arc (103.134022:54.246113:3.696701);
\draw [thick](3.000000,3.000000) arc (135.000000:45.000000:4.242641);
\node [vertex] at (9.000000,3.000000) {};
\draw (1.968750-3.172298,0.000000) arc (180:0:3.172298);
\draw (10.031250+3.172298,-0.000000) arc (0:180:3.172298);
\draw (6, 0) -- (6, 5);
\draw (-1.5, 0) -- (13.5, 0);
\draw [thick](6.0-2.487468,0.000000) arc (180:0:2.487468);
\end{tikzpicture}
\]
\caption{An isosceles hyperbolic triangle with a hyperbolic orthocenter. The othocenter is a unique line orthogonal to all altitudes.}
\label{fig:orthocenter}
\end{figure}

In all cases we have $\ch_a=[0:\cos\beta:-\cos\gamma]$, $\ch_b=[\cos\alpha:0:-\cos\gamma]$, $\ch_c=[\cos\alpha:-\cos\beta:0]$, and $\gH=(\frac{1}{\cos\alpha}:\frac{1}{\cos\beta}:\frac{1}{\cos\gamma})$.
Theorem~\ref{type1} gives the expression
$$
X = \tan^2\alpha+\tan^2\beta+\tan^2\gamma + 6 + 
2\frac{\cos^2\alpha+\cos^2\beta+\cos^2\gamma}{\cos\alpha\cos\beta\cos\gamma}.
$$
for the orthocenter. Thus, for an acute triangle $\gH$ is elliptic, but this is not a criterion.

\section{Menelaus's Theorem}
We prove Menelaus's Theorem in its most general form.
\begin{theorem}[Menelaus's Theorem]
Let $\ca, \cb,\cc$ be three distinct non-collinear cycles. Let $\cn_a\in\gA$, 
$\cn_b\in\gB$, $\cn_c\in\gC$ be three cycles distinct from $\ca$, $\cb$, $\cc$. Then 
the pencils $\GG(\ca,\cn_a)$, $\GG(\cb,\cn_b)$, $\GG(\cc, \cn_c)$
are concurrent iff
$$
(\cb,\cc;\cn_a)(\cc,\ca;\cn_b)(\ca,\cb;\cn_c) = -1.
$$
Moreover, if $\cn$ is the common cycle and $\cn_a$ is cyclic, then
$$
(\ca,\cn_a;\cn) = (\cb,\cn_a;\cc)(\ca,\cb;\cn_c).
$$

\end{theorem}
When $\ca$, $\cb$, $\cc$ form a M\"obius triangle with positive inner angles then 
the condition in the Menelaus's Theorem has its usual trigonometric form
$$
\frac{\sin\angle(\cb, \cn_a)}{\sin\angle(\cc, \cn_a)}
\frac{\sin\angle(\cc, \cn_b)}{\sin\angle(\ca, \cn_b)}
\frac{\sin\angle(\ca, \cn_c)}{\sin\angle(\cb, \cn_c)}
=-1.
$$

\begin{proof}
Let $\lambda=(\cb,\cc;\cn_a)$, $\mu=(\cc,\ca;\cn_b)$, $\nu=(\ca,\cb;\cn_c)$.
Then 
$\GG(\ca,\cn_a)=(0:\lambda:1)$, $\GG(\cb,\cn_b)=(1:0:\mu)$, $\GG(\cc, \cn_c)=(\nu:1:0)$. Therefore, $\GG(\ca,\cn_a)$, $\GG(\cb,\cn_b)$, $\GG(\cc, \cn_c)$ are concurrent iff
$$
\left|
\begin{array}{ccc}
0 & \lambda & 1 \\
1 & 0 & \mu\\
\nu & 1 & 0
\end{array}
\right|=\lambda\mu\nu+1 =0.
$$
Let $\ca,\cb,\cc\approx L, M, N$. Then $\cn_a\approx (\cn_a,\cb;\cc)(M-\lambda N)$, and $\cn\sim L-\nu M+\lambda\nu N$. Therefore,
$(\ca,\cn_a;\cn) = \nu (\cb,\cn_a;\cc)$.
\end{proof}

Note that when all cevians $\cn_a,\cn_b,\cn_c$ are real the cycle $\cn$ can be virtual. 
We identify its type. We have $\cn=[u:v:w]$ where
$$
\frac{w}{v} = -(\cb,\cc;\cn_a),\quad
\frac{u}{w} = -(\cc,\ca;\cn_b),\quad
\frac{v}{u} = -(\ca,\cb;\cn_c).
$$

\begin{theorem}\label{menelaus-type}
Let $\ca,\cb,\cc$ be three non-collinear oriented cycles. Let $\cn=[u:v:w]$, and 
$$
Y = u^2+v^2+w^2 + 2uv \la\ca,\cb \ra+2uw\la\ca,\cc\ra + 2vw\la \cb,\cc \ra.
$$
Then $\cn$ is elliptic iff $Y>0$, parabolic iff $Y=0$, and hyperbolic iff $Y<0$.
\end{theorem}
\begin{proof}
Let $\ca,\cb,\cc\approx L, M, N$. Then $\cn\sim uL+vM+wN$ and
\begin{gather*}
\det (uL+vM+wN)= -\la uL+vM+wN, uL+vM+wN \ra = \\
=-u^2-v^2-w^2 - 2uv \la\ca,\cb \ra-2uw\la\ca,\cc\ra - 2vw\la \cb,\cc \ra.
\end{gather*}
Then apply Proposition~\ref{cycle-type}.
\end{proof}

Consider now the standard geometries.
Assume that the cevians $\cn_a, \cn_b, \cn_c$ satisfy the condition on the Menelaus's Theorem. Then
\begin{itemize}
\item for a spherical triangle $\cn_a, \cn_b, \cn_c$ intersect $\ca,\cb,\cc$ in three collinear points defining a new line $\cn$.

\item for a Euclidean triangle $\cn_a, \cn_b, \cn_c$ either intersect $\ca,\cb,\cc$ in three collinear points, or two cevians intersect opposite sides and the intersection points define a line parallel to the third cevian, or every cevian is parallel to the opposite side; in this case $\cn$ is the line at infinity.  

\item for a hyperbolic triangle there are even more cases. When $\ca$ and $\cn_a$
do not intersect let $\cp_a$ be their common perpendicular; define $\cp_b$ and $\cp_c$ analogously. Menelaus's Theorem defines a generalized cycle $\cn$ such that exactly one of the following cases holds (Figure~\ref{fig:menelaus}):
\begin{itemize}
\item[$(i)$] $\cn_a$, $\cn_b$, $\cn_c$ intersect the opposite sides in three collinear points lying on $\cn$.
\item[$(ii)$] one of the lines $\cp_a$, $\cp_b$, $\cp_c$ exists, $\cn$ is orthogonal to it and passes through the intersection points of the remaining two cevians with the opposite sides.

\item[$(iii)$] two of the lines $\cp_a$, $\cp_b$, $\cp_c$ exist, $\cn$ is orthogonal to both of them, and passes through the intersection point of the third cevian with the opposite side.

\item[$(iv)$] $\cp_a$, $\cp_b$, $\cp_c$ exist and are orthogonal to $\cn$.

\item[$(v)$] $\cp_a$, $\cp_b$, $\cp_c$ exist and approach the same point at infinity. In this case $\cn$ is parabolic and coincides with this point.

\item[$(vi)$] $\cp_a$, $\cp_b$, $\cp_c$ exist and are concurrent. In this case $\cn$ is virtual and corresponds to the intersection point of $\cp_a$, $\cp_b$, $\cp_c$.

\end{itemize}
\end{itemize}

\begin{figure}
\[
\begin{tikzpicture}
\tikzstyle{vertex}=[circle, fill, inner sep=1.5pt]
\draw[thick] (3.000000,1.000000) arc (170.134193:136.735705:5.836309);
\draw (3.000000,1.000000) arc (170.134193:126.735705:5.836309);
\node [vertex] at (3.000000,1.000000) {};
\node [vertex] at (4.500000,4.000000) {};
\draw [thick] (4.500000,4.000000) arc (68.404690:44.215175:4.301970);
\node [vertex] at (6.000000,3.000000) {};
\draw [thick] (3.000000,1.000000) arc (160.559965:86.820170:3.004626);

\node [vertex] at (4.883253,3.826157) {};
\node [vertex] at (4.727012,2.793534) {};
\node [vertex] at (5.007938,4.478780) {};

\draw (2.891705,0.000000) arc (180.000000:110.000000:4.671176);
\draw (4.940551,0.000000) arc (0.000000:15.000000:18.379348);
\draw (6.700885,0.000000) arc (0.000000:45.000000:6.770899);
\draw [thick] (4.549989,0.000000) arc (180.000000:167.000000:22.130393);

\draw [thick](11.000000,1.000000) arc (170.134193:136.735705:5.836309);
\node [vertex] at (11.000000,1.000000) {};
\node [vertex] at (12.500000,4.000000) {};
\draw [thick](12.500000,4.000000) arc (68.404690:44.215175:4.301970);
\node [vertex] at (14.000000,3.000000) {};
\draw [thick](11.000000,1.000000) arc (160.559965:86.820170:3.004626);

\draw (10.875815,0.000000) arc (180.000000:100.000000:4.088351);
\draw (9.983443,0.000000) arc (180.000000:90.000000:4.437225);
\draw (9.797171,0.000000) arc (180.000000:60.000000:3.172122);
\draw [thick] (10.338331,0.000000) arc (180.000000:90.000000:3.796414);
\draw [dashed] (13.303363,0.000000) arc (0.000000:35.000000:8.252262);

\node [vertex] at (13.161829,3.669631) {};
\node [vertex] at (11.685618,2.900782) {};

\draw [thick](11.000000,-11.000000) arc (170.134193:136.735705:5.836309);
\node [vertex] at (11.000000,-11.000000) {};
\node [vertex] at (12.500000,-8.000000) {};
\draw [thick](12.500000,-8.000000) arc (68.404690:44.215175:4.301970);
\node [vertex] at (14.000000,-9.000000) {};
\draw [thick](11.000000,-11.000000) arc (160.559965:86.820170:3.004626);

\node [vertex] at (12.390873,-8.739522) {};

\draw (10.498851,-12.000000) arc (180.000000:0.000000:1.248281);
\draw (10.049144,-12.000000) arc (180.000000:90.000000:4.489594);
\draw (11.716888,-12.000000) arc (180.000000:90.000000:3.112550);

\draw [dashed](11.827047,-12.000000) arc (180.000000:150.000000:9.709216);
\draw [dashed](13.208399,-12.000000) arc (0.000000:45.000000:6.910519);
\draw [dashed](13.838872,-12.000000) arc (0.000000:70.000000:4.394830);

\draw [thick](11.000000,-5.000000) arc (170.134193:136.735705:5.836309);
\node [vertex] at (11.000000,-5.000000) {};
\node [vertex] at (12.500000,-2.000000) {};
\draw [thick](12.500000,-2.000000) arc (68.404690:44.215175:4.301970);
\draw (12.500000,-2.000000) arc (68.404690:0:4.301970);
\node [vertex] at (14.000000,-3.000000) {};
\draw [thick](11.000000,-5.000000) arc (160.559965:86.820170:3.004626);

\draw (10.746949,-6.000000) arc (180.000000:0.000000:2.102409);
\draw (9.857514,-6.000000) arc (180.000000:80.000000:4.348695);
\draw (11.491933,-6.000000) arc (180.000000:80.000000:3.048244);

\draw [dashed](13.868569,-6.000000) arc (180.000000:90.000000:1.658797);
\draw [dashed](13.483185,-6.000000) arc (0.000000:25.000000:12.716299);
\draw [dashed](13.350413,-6.000000) arc (0.000000:90.000000:3.310007);
\draw [thick](12.421675,-6.000000) arc (180.000000:00.000000:1.109852);

\draw [thick](3.000000,-5.000000) arc (170.134193:136.735705:5.836309);
\node [vertex] at (3.000000,-5.000000) {};
\node [vertex] at (4.500000,-2.000000) {};
\draw [thick](4.500000,-2.000000) arc (68.404690:44.215175:4.301970);
\draw (4.500000,-2.000000) arc (68.404690:0:4.301970);
\node [vertex] at (6.000000,-3.000000) {};
\draw [thick](3.000000,-5.000000) arc (160.559965:86.820170:3.004626);

\draw (2.784682,-6.000000) arc (180.000000:0.000000:2.429804);
\draw (1.857514,-6.000000) arc (180.000000:70.000000:4.348695);
\draw (3.275946,-6.000000) arc (180.000000:60.000000:3.013977);
\node [vertex] at (6.807959,-4.165663) {};

\draw [dashed](5.483185,-6.000000) arc (0.000000:25.000000:12.716299);
\draw [dashed](4.833435,-6.000000) arc (0.000000:110.000000:2.390234);
\draw [thick](3.504482,-6.000000) arc (180.000000:0.000000:2.161019);

\draw [thick](3.000000,-11.000000) arc (170.134193:136.735705:5.836309);
\node [vertex] at (3.000000,-11.000000) {};
\node [vertex] at (4.500000,-8.000000) {};
\draw [thick](4.500000,-8.000000) arc (68.404690:44.215175:4.301970);
\draw (4.500000,-8.000000) arc (68.404690:0:4.301970);
\node [vertex] at (6.000000,-9.000000) {};
\draw [thick](3.000000,-11.000000) arc (160.559965:86.820170:3.004626);

\draw (2.740230,-12.000000) arc (180.000000:0.000000:2.054666);
\draw (1.724297,-12.000000) arc (180.000000:70.000000:4.270004);
\draw (3.642277,-12.000000) arc (180.000000:70.000000:3.087483);
\draw [dashed](5.663801,-12.000000) arc (180.000000:90.000000:1.994841);
\draw [dashed](5.674580,-12.000000) arc (0.000000:10.000000:28.353994);
\draw [dashed](5.678874,-12.000000) arc (0.000000:70.000000:4.010042);
\node [vertex] at (5.673873,-12.0) {};

\draw (1, 0) -- (16, 0);
\draw (1, -6) -- (16, -6);
\draw (1, -12) -- (16, -12);
\end{tikzpicture}
\]
\caption{Various cases $(i)$--$(vi)$ of Menelaus's Theorem for a hyperbolic triangle. Every cevian either intersects the opposite side of the triangle or admits a common perpendicular (dashed line).}
\label{fig:menelaus}
\end{figure}
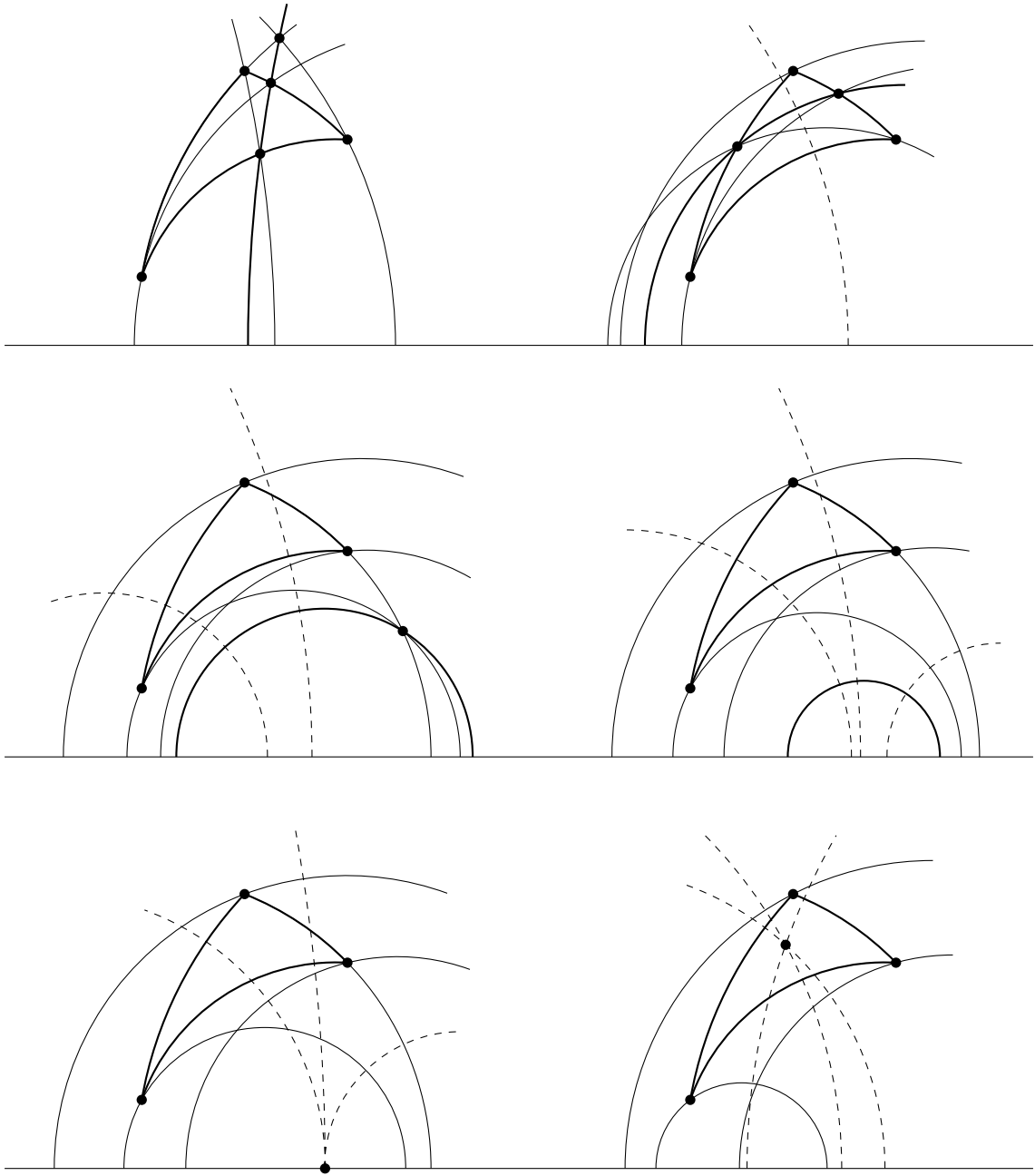

\section{Duality}
Let $\ca,\cb,\cc$ be three non-collinear cycles. Let $V\subset\Herm(2)$ be the corresponding subspace of dimension $3$. If $V$ is non-degenerate there exists a natural bijection between $\GG(\ca,\cb,\cc)$ and $\PP(\ca,\cb,\cc)$.
Indeed, $\GG(\ca,\cb,\cc)$ is in bijection with $1$-dimensional subspaces of $V$, i.e. with the Grassmannian $\Gr_1(V)$, and $\PP(\ca,\cb,\cc)$ is, in turn, in bijection with $\Gr_2(V)$. The orthogonal complement $W\mapsto W^\perp$ gives a bijection between $\Gr_2(V)$ and $\Gr_1(V)$.  Therefore, we obtain an identification between $\GG(\ca,\cb,\cc)$ and $\PP(\ca,\cb,\cc)$. This works for spherical and hyperbolic geometries, but not for Euclidean. We use notation $\gP^\ast$, and $\ca^\ast$ for this map. Trivially, we have $\gP^{\ast\ast}=\gP$, $\ca^{\ast\ast}=\ca$ for all $\gP$ and $\ca$. We call $\gP^\ast$ the \textit{complement} of $\gP$, and similarly for $\ca^\ast$.

The complement can be written in coordinate form. That is, for every generalized triangle  
there exists a projective transformation $T\in\PGL(3;\R)$ that makes the identification; that is such that $[u:v:w]=(x:y:z)^\ast$ iff 
$\left(\begin{smallmatrix}u\\ v\\ w\end{smallmatrix}\right)=T
\left(\begin{smallmatrix}x\\ y\\ z\end{smallmatrix}\right)$.
We compute the transformation $T$.

\begin{prop}
Let $\ca,\cb,\cc$ be non-collinear cycles generating a non-degenerate subspace in $\Herm(2)$. Then the transformation $T$ is given by the formula
$$
T=\left(\begin{array}{ccc}
1-\la\cb,\cc\ra^2 &  \la\ca,\cc\ra\la\cb,\cc\ra-\la\ca,\cb\ra & \la\ca,\cb\ra\la\cb,\cc\ra-\la\ca,\cc\ra\\
\la\ca,\cc\ra\la\cb,\cc\ra-\la\ca,\cb\ra & 1-\la\ca,\cc\ra^2 &
\la\ca,\cb\ra\la\ca,\cc\ra-\la\cb,\cc\ra\\
\la\ca,\cb\ra\la\cb,\cc\ra-\la\ca,\cc\ra&
\la\ca,\cb\ra\la\ca,\cc\ra-\la\cb,\cc\ra& 1- \la\ca,\cb\ra^2
\end{array}\right)
$$
up to a nonzero constant. The opposite map is given by the Gram matrix
$$
\Gamma=
\left(
\begin{matrix}
1 & \la\ca,\cb\ra & \la\ca,\cc\ra \\
\la\ca,\cb\ra & 1 & \la\cb,\cc\ra \\
\la\ca,\cc\ra & \la\cb,\cc\ra & 1\\
\end{matrix}
\right).
$$
\end{prop}
\begin{proof}
By a straight-forward computation we establish the identity 
$T\Gamma=\det \Gamma I$.

Let $\ca, \cb, \cc\approx L, M, N$. Let $\gP=(x:y:z)$ be a pencil. Let $\cn_a=[0:z:-y]$, $\cn_b=[z:0:-x]$, $\cn_c=[y:-x:0]$ be cevians lying in $\gP$. We compute the only generalized cycle $\cn\in\gP^\perp$. Let $\cn=[u:v:w]$. Then we have that conditions $\cn\perp \cn_a, \cn_b, \cn_c$ are equivalent to the system of linear equations
\begin{gather*}
\begin{cases}
\la uL+vM+wN,zM-yN\ra= uz\la\ca,\cb\ra - uy\la\ca,\cc\ra
+vz-vy\la\cb,\cc\ra+wz\la\cb,\cc\ra - wy=0,\\
\la uL+vM+wN,zL-xN\ra=uz - ux\la\ca,\cc\ra+vz\la\ca,\cb\ra -vx\la\cb,\cc\ra
+wz\la\ca,\cc\ra-wx=0,\\
\la uL+vM+wN,yL-xM\ra=uy -ux\la\ca,\cb\ra+vy\la\ca,\cb\ra - vx
+wy\la\ca,\cc\ra-wx\la\cb,\cc\ra=0.
\end{cases}
\end{gather*}
Or, in equivalent form,
$$
\left(
\begin{array}{ccc}
z\la\ca,\cb\ra - y\la\ca,\cc\ra & z-y\la\cb,\cc\ra & z\la\cb,\cc\ra - y\\
z - x\la\ca,\cc\ra & z\la\ca,\cb\ra -x\la\cb,\cc\ra & z\la\ca,\cc\ra-x\\
y -x\la\ca,\cb\ra & y\la\ca,\cb\ra - x & y\la\ca,\cc\ra-x\la\cb,\cc\ra
\end{array}
\right)
\left(
\begin{array}{c}
u\\
v\\
w
\end{array}
\right)=0
$$
We can see that the matrix on the left hand side factors as
$$
\left(
\begin{array}{ccc}
0 & z & -y \\
z & 0 & -x \\
y & -x & 0\\
\end{array}
\right)
\left(
\begin{array}{ccc}
1 & \la\ca,\cb\ra & \la\ca,\cc\ra \\
\la\ca,\cb\ra & 1 & \la\cb,\cc\ra \\
\la\ca,\cc\ra & \la\cb,\cc\ra & 1\\
\end{array}
\right).
$$
The first factor is a rank $2$ matrix with its kernel generated by the vector 
$\left(\begin{smallmatrix}x\\ y\\ z\end{smallmatrix}\right)$, and second is the Gram matrix $\Gamma$, which is assumed to be nondegenerate.

Therefore, we have that $\left(\begin{smallmatrix}u\\ v\\ w\end{smallmatrix}\right)$ up to a nonzero constant is equal to $\Gamma^{-1}\left(\begin{smallmatrix}x\\ y\\ z\end{smallmatrix}\right)$, or to
$T\left(\begin{smallmatrix}x\\ y\\ z\end{smallmatrix}\right)$.
\end{proof}
When the subspace $V$ is spherical, then all types are elliptic. In the hyperbolic case the identification flips the type.
\begin{prop}
Let $\ca,\cb,\cc$ generate a hyperbolic (i.e. of signature $(2,1)$) subspace $V$. Then
a pencil $\gP\in\PP(\ca,\cb,\cc)$  is elliptic, parabolic, or hyperbolic
iff the cycle $\gP^\ast$ is hyperbolic, parabolic, or elliptic respectively. 
\end{prop}
\begin{proof}
We simply compute the signatures of the corresponding subspaces.
Let $\gP$ generate a subspace $W\subset V$. If, for example, $\gP$ is hyperbolic, 
then $W$ is indefinite, and $\sig W=(1,1)$. But then $W^\perp$ is a positive-definite line in $V$ that corresponds to an elliptic cycle $\ca=\gP^\ast$. The remaining cases are considered analogously.

Alternatively, we can make use of Theorems~\ref{type1} and \ref{menelaus-type}. The values $X$, $Y$ satisfy the 
identities
$$
X = \left(\begin{smallmatrix} x\\y\\z \end{smallmatrix}\right)^\top T
\left(\begin{smallmatrix} x\\y\\z \end{smallmatrix}\right),\quad 
Y = \left(\begin{smallmatrix} u\\v\\w \end{smallmatrix}\right)^\top \Gamma
\left(\begin{smallmatrix} u\\v\\w \end{smallmatrix}\right).
$$
We have also 
$$
\left(\begin{smallmatrix}u\\ v\\ w\end{smallmatrix}\right)=cT
\left(\begin{smallmatrix}x\\ y\\ z\end{smallmatrix}\right)
$$
for some $c\neq0$. This implies that
$$
X = \left(\begin{smallmatrix} x\\y\\z \end{smallmatrix}\right)^\top T
\left(\begin{smallmatrix} x\\y\\z \end{smallmatrix}\right)=
c^2\left(\begin{smallmatrix} u\\v\\w \end{smallmatrix}\right)^\top \Gamma T\Gamma
\left(\begin{smallmatrix} u\\v\\w \end{smallmatrix}\right)=
c^2\det\Gamma\cdot\left(\begin{smallmatrix} u\\v\\w \end{smallmatrix}\right)^\top \Gamma 
\left(\begin{smallmatrix} u\\v\\w \end{smallmatrix}\right)=
c^2\det\Gamma\cdot Y.
$$
Since $\det\Gamma<0$ we have that $X$ and $Y$ always have different signs or both are equal to zero.
\end{proof}

We also identify all generalized triangles generating degenerate subspaces.
These are exactly the Euclidean triangles.

\begin{theorem}
Three non-collinear cycles $\ca,\cb,\cc$ generate a degenerate subspace in $\Herm(2)$ iff $\ca,\cb,\cc$ share a common point.
\end{theorem}
\begin{proof}
A degenerate 3-dimensional subspace of $\Herm(2)$ has signature $(2,0)$ only. Therefore, all pencils are elliptic or parabolic and we also have $|\la\ca,\cb\ra|\leq 1$,
$|\la\cb,\cc\ra|\leq 1$, $|\la\ca,\cc\ra|\leq 1$. Let $\la\ca,\cb\ra=-\cos\alpha$, 
$\la\ca,\cc\ra=-\cos\beta$, $\la\cb,\cc\ra=-\cos\gamma$, for some $\alpha,
\beta,\gamma\in[0,\pi]$. Then $\det\Gamma=0$ implies
$$
\cos^2\alpha+\cos^2\beta+\cos^2\gamma+2\cos\alpha\cos\beta\cos\gamma=1.
$$
Consider this and solve it as a quadratic equation for $\cos\gamma$. The discriminant is equal to
$$
4\cos^2\alpha\cos^2\beta - 4(\cos^2\alpha+\cos^2\beta-1) = 
4(1-\cos^2\alpha)(1-\cos^2\beta) = 4\sin^2\alpha\sin^2\beta.
$$
Therefore, 
$$
\cos\gamma = -\cos\alpha\cos\beta\pm\sin\alpha\sin\beta = -\cos(\alpha\pm\beta).
$$
Every two cycles intersect, so there exists a proper triangle $\triag ABC$ bounded by some segments of $\ca$, $\cb$, $\cc$. Change the orientation of the lines if necessary to assume that triangle is properly oriented. Then its inner angles are $\alpha$, $\beta$, $\gamma$. Then the condition $\cos\gamma=-\cos(\alpha\pm\beta)$ implies that one of the following identities holds: $\alpha+\beta+\gamma=\pi$, $\alpha=\beta+\gamma-\pi$, $\beta=\alpha+\gamma-\pi$, or $\gamma=\alpha+\beta-\pi$.

In the first case if we assume that $\ca$ and $\cb$ are straight lines, then it would immediately follow that $\cc$ is also a straight line and then all cycles pass through $\infty$. In the remaining cases Theorem~\ref{a=b+c-Pi} implies that one of the vertices $A, B, C$ belongs to all cycles.

Conversely, assume that $\ca,\cb,\cc$ have a common point. Without loss of generation, assume it is $\infty$. Then we have a generalized Euclidean triangle with angles $\alpha,\beta,\gamma$ satisfying $\alpha+\beta+\gamma=\pi$. Then
$\cos\gamma=-\cos(\alpha+\beta)$ and this implies the identity.
\end{proof}

\section{Isogonal conjugation}
Finally, we can also generalize the well-known partial transformation of isogonal
conjugation to M\"obius triangles. Let $\cn_a, \cn_b,\cn_c$ be three cevians satisfying the conditions of Ceva's or Menelaus's Theorem. Then we can flip them around the corresponding bisectors and get new cevians $\cm_a, \cm_b, \cm_c$. But since $(\ca,\cb;\cm_c)=(\ca,\cb;\cn_c)^{-1}$, and etc; we have that $\cm_a, \cm_b, \cm_c$ also satisfy the conditions of the same theorem (Figure~\ref{fig:isogonal}).

\begin{figure}[h]
\[
\begin{tikzpicture}
\tikzstyle{vertex}=[circle, fill, inner sep=1.5pt]
\draw [thick](1.410000,6.495000) arc (50.176327:28.661241:17.528410);
\node [vertex] at (1.410000,6.495000) {};
\node [vertex] at (5.565000,1.440000) {};
\draw [thick](1.410000,6.495000) arc (95.176327:65.833604:17.053332);
\node [vertex] at (9.930000,5.070000) {};
\draw [thick](5.565000,1.440000) arc (148.661241:110.833604:8.757112);
\draw (5.565000,1.440000) arc (18.577333:28.577333:36.610349);
\draw (1.410000,6.495000) arc (64.456354:38.456354:16.215145);
\draw (9.930000,5.070000) arc (79.924268:116.924268:12.207683);
\draw (5.565000,1.440000) arc (158.745149:125.745149:11.191805);
\draw (1.410000,6.495000) arc (80.896300:50.896300:16.060508);
\draw (9.930000,5.070000) arc (96.742940:136.742940:9.695034);

\draw [dashed](5.565000,1.440000) arc (178.661241:167.661241:30.310991);
\draw [dashed](1.410000,6.495000) arc (72.676327:42.676327:15.971667);
\draw [dashed](9.930000,5.070000) arc (88.333604:128.333604:10.691041);

\node [vertex] at (4.275,4.74) {};
\node [vertex] at (7.2342,4.34669) {};

\end{tikzpicture}
\]
\caption{Isogonal conjugation for pencils in a M\"obius triangle.}
\label{fig:isogonal}
\end{figure}
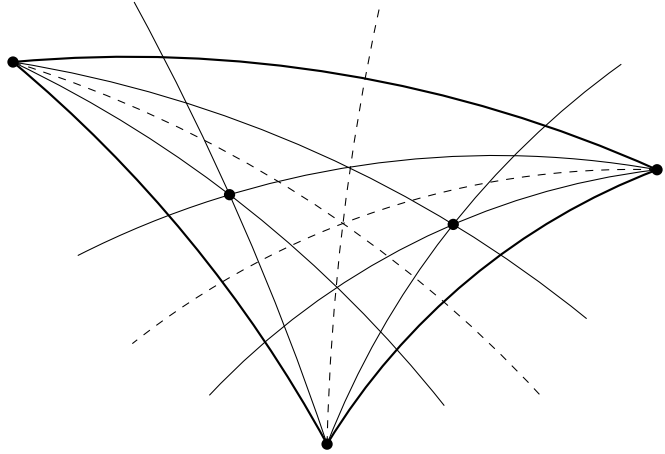

Therefore, we have a partial transformations on the  planes $\PP(\ca,\cb,\cc)$ and $\GG(\ca,\cb,\cc)$ -- the \textit{isogonal conjugation}; it is defined for both pencils and cycles associated with the triangle. We denote the images by $\gP^\#$ and $\cn^\flat$ respectively. The following statements are trivial.

\begin{theorem}
Isogonal conjugation is not defined on the pencils $\gA, \gB,\gC$.
If $\ca\in\gP$ then $\gP^\#=\gA$, and similarly for $\cb,\cc$.
For the remaining pencils isogonal conjugation is an involution: $\gP^{\#\#}=\gP$.
Analogously, isogonal conjugation is not defined on the cycles $\ca,\cb,\cc$.
If $\cn\in\gA$ then $\cn^\flat=\ca$, and similarly for $\gB$, $\gC$.
For the remaining cycles isogonal conjugation is an involution: $\cn^{\flat\flat}=\cn$. 
\end{theorem}

\begin{prop}
In trilinear coordinates the isogonal conjugation transformations $\gP\mapsto\gP^\#$ and $\cn\mapsto\cn^\flat$ are given by the maps
$$
(x:y:z)\mapsto(\tfrac{1}{x}:\tfrac{1}{y}:\tfrac{1}{z}),
$$
and
$$
[u:v:w]\mapsto[\tfrac{1}{u}:\tfrac{1}{v}:\tfrac{1}{w}].
$$
\end{prop}

\begin{corollary}
The only self-conjugate pencils are the incenter $\gI$ and the excenters 
$\gI_a$, $\gI_b$, $\gI_c$.
\end{corollary}
Similarly, the only self-conjugate cycles are $[1:1:1]$, $[-1:1:1]$, $[1:-1:1]$, $[1:1:-1]$, but they don't have any special names.

Isogonal conjugation on the M\"obius plane, of course, generalizes the same concept for the standard geometries. For the spherical and hyperbolic geometries the planes $\GG(\ca,\cb,\cc)$ and $\PP(\ca,\cb,\cc)$ are both naturally identified with the space $\LL$ of linear pencils, but the identification gives two different partial transformations on $\LL$, so when taking isogonal conjugate we must remember in which plane (pencil or cycle) we are. That is why we have introduced two types of notation.

We have the following negative results:
\begin{itemize}
\item[$(i)$] isogonal conjugation does not preserve the type. An elliptic pencil can be conjugate to a hyperbolic one, and a real cycle to a virtual.

\item[$(ii)$] isogonal conjugation does not commute with the complement, i.e. $(\gP^\ast)^\flat\neq(\gP^\#)^\ast$ in general.

\item[$(iii)$] isogonal conjugation cannot be computed pointwise, i.e. 
$\gP^\#\neq\{\ca^\flat|\ca\in\gP\}$.
\end{itemize}

\end{document}